\documentclass[12pt]{article}

\usepackage[utf8]{inputenc}
\usepackage[T1]{fontenc}
\usepackage{amsthm,amsmath,amssymb,amsfonts,dsfont,mathtools}
\usepackage[colorlinks,citecolor=blue,urlcolor=blue]{hyperref}
\usepackage{xcolor}
\usepackage{tikz}
\usepackage{longtable}
\usepackage{version}
\usepackage{natbib}
\usepackage[babel=true]{csquotes}
\usepackage{geometry}

\newcommand{\1}{\mathds{1}}
\newcommand{\edge}{y}
\newcommand{\MA}{\by}
\newcommand{\MAO}{\MA^\text{\rm o}}
\newcommand{\MAM}{\MA^\text{\rm m}}

\newcommand{\block}{\mathcal{Q}}
\newcommand{\dyad}{\mathcal{D}}

\newcommand{\node}{\mathcal{N}}

\newcommand{\Pbb}{\mathbb{P}}

\numberwithin{equation}{section}
\theoremstyle{plain}
\newtheorem{thm}{Theorem}[section]
\newtheorem{dof}[thm]{Definition}
\newtheorem{proposition}[thm]{Proposition}
\newtheorem{corollaire}[thm]{Corollary}
\newtheorem{lemme}[thm]{Lemma}
\newtheorem{property}[thm]{Property}
\theoremstyle{remark}
\newtheorem{rem}[thm]{Remark}
\newtheorem{example}{Example}
\newcommand{\proofbegin}{\paragraph{Proof.}\hspace{1em}\\}
\newcommand{\proofend}{\begin{flushright}
$\square$
\end{flushright}}

\newcommand{\dd}{d}
\newcommand{\g}{g}
\newcommand{\q}{\mathcal{Q}}
\newcommand{\ii}{i}
\newcommand{\jj}{j}
\newcommand{\kk}{q}
\newcommand{\el}{\ell}

\newcommand{\n}{n}
\newcommand{\vrai}{\star}
\newcommand{\al}{\pi}
\newcommand{\als}{\al^\vrai}
\newcommand{\alkl}{\al_{\kk\el}}
\newcommand{\piql}{\pi_{q \el}}
\newcommand{\alskl}{\als_{\kk\el}}
\newcommand{\bal}{\boldsymbol{\al}}
\newcommand{\bals}{\boldsymbol{\al}^\vrai}

\newcommand{\thetaa}{\theta}
\newcommand{\Thetaa}{\Theta}
\newcommand{\btheta}{\boldsymbol{\thetaa}}
\newcommand{\hbtheta}{\widehat{\btheta}}
\newcommand{\bthetas}{\boldsymbol{\thetaa}^\vrai}
\newcommand{\bTheta}{\boldsymbol{\Thetaa}}
\newcommand{\pii}{\alpha}
\newcommand{\bpi}{\boldsymbol{\pii}}
\newcommand{\pik}{\pii_\kk}

\newcommand{\bpis}{\boldsymbol{\pii}^\vrai}

\newcommand{\zjl}{\z_{\jj\el}}

\newcommand{\y}{y}

\newcommand{\by}{\mathbf{\y}}
\newcommand{\byo}{\mathbf{\y^o}}
\newcommand{\br}{\mathbf{r}}

\newcommand{\yij}{\y_{\ii\jj}}

\newcommand{\Y}{Y}
\newcommand{\R}{R}

\newcommand{\z}{z}
\newcommand{\bfz}{\mathbf{\z}}
\newcommand{\bz}{\bfz}

\newcommand{\zik}{\z_{\ii\kk}}
\newcommand{\ziq}{\z_{\ii q}}
\newcommand{\rrij}{r_{\ii \jj}}
\newcommand{\bzs}{\bfz^\vrai}

\newcommand{\zj}{\z_{\jj}}
\newcommand{\zi}{\z_{\ii}}

\newcommand{\Sal}{\boldsymbol{S}^{\vrai}}

\newcommand{\Salkl}{S^\vrai_{\kk\el}}
\newcommand{\Salmin}{S^\vrai_{\min}}
\newcommand{\Salmax}{S^\vrai_{\max}}

\newcommand{\mcL}{\mathcal{L}}

\newcommand{\Lc}{\mcL_{co}}

\newcommand{\prob}{p}
\newcommand{\Prob}{\mathbb{P}}
\newcommand{\Esp}{\mathbb{E}}
\newcommand{\Var}{\mathbb{V}}
\newcommand{\dens}{\varphi}
\newcommand{\norm}{\psi}
\newcommand{\normp}{\norm'}
\newcommand{\normpm}{(\normp)^{-1}}
\newcommand{\mcZ}{\mathcal{Z}}

\newcommand{\LL}{\mcL}

\newcommand{\Lcs}{\Lc^{\vrai}}

\newcommand{\sommecolonne}{+}
\newcommand{\zsumk}{\z_{\sommecolonne\kk}}
\newcommand{\zsumq}{\z_{\sommecolonne q}}

\newcommand{\zs}{\z^{\vrai}}

\newcommand{\zssumk}{\zs_{\sommecolonne\kk}}

\newcommand{\Sig}{\Sigma}

\newcommand{\mcN}{\mathcal{N}}

\newcommand{\hpi}{\hat{\pii}}
\newcommand{\hbpi}{\widehat{\bpi}}

\newcommand{\pizk}{\widehat{\pii}_{\kk}(\bz)}
\newcommand{\pizl}{\widehat{\pii}_{\el}(\bz)}

\newcommand{\pizq}{\widehat{\pii}_{q}(\bz)}

\newcommand{\pizql}{\widehat{\pi}_{q\el}(\bz)}
\newcommand{\hal}{\widehat{\al}}
\newcommand{\hbal}{\widehat{\bal}}

\newcommand{\kp}{\kk'}

\def\AArm{\fam0 }%\tenrm}%
\def\R{{\AArm I\!R}}%
\def\un{{\AArm 1\!I}}%

\newcommand{\RQ}{\R}

\newcommand{\RQbz}{\RQ(\bz)}

\newcommand{\Q}{Q}

\newcommand{\transpose}{{^{T}}}
\newcommand{\vareps}{\varepsilon}

\newcommand{\Lamb}{\Lambda}

\newcommand{\Lambtilde}{\tilde{\Lamb}}
\newcommand{\lp}{\el'}

\newcommand{\bary}{\bar{y}}

\newcommand{\hy}{\widehat{\y}}

\newcommand{\hykl}{\hy_{\kk\el}}

\newcommand{\hyklz}{\hy_{\kk\el}(\bz)}
\newcommand{\hyqlz}{\hy_{q\el}(\bz)}
\newcommand{\baral}{\bar{\al}}
\newcommand{\baralkl}{\baral_{\kk\el}}

\newcommand{\barykl}{\bary_{\kk\el}}

\newcommand{\baryklz}{\bary_{\kk\el}(\bz)}

\newcommand{\neighborsize}{\kappa}
\newcommand{\vmin}{\underline{\sigma}}
\newcommand{\vmax}{\bar{\sigma}}

\newcommand{\tnd}{t_{\n\dd}}
\newcommand{\tn}{t_{\n}}

\newcommand{\Fn}{LR}
\newcommand{\G}{ELR}

\newcommand{\bigO}{\mathcal{O}}
\newcommand{\smallO}{o}
\newcommand{\Om}{\Omega}

\DeclareMathOperator{\Diag}{Diag}
\DeclareMathOperator{\Trace}{Tr}
\DeclareMathOperator{\Diam}{Diam}
\DeclareMathOperator{\Symmetric}{Sym}
\DeclareMathOperator{\KL}{KL}

\newcommand{\bY}{\mathbf{\Y}}
\newcommand{\bH}{\mathbf{H}}
\newcommand{\bthetaEMV}{\widehat{\btheta}_{MLE}}
\newcommand{\bpiEMV}{\widehat{\bpi}_{MLE}}

\newcommand{\balEMV}{\widehat{\bal}_{MLE}}
\newcommand{\setQ}{\mathbb{Q}}
\newcommand{\Jvar}{J}
\newcommand{\mcH}{\mathcal{H}}
\newcommand{\bthetavar}{\widehat{\btheta}_{var}}
\newcommand{\balvar}{\widehat{\bal}_{var}}

\newcommand{\bpivar}{\widehat{\bpi}_{var}}
\newcommand{\mcQ}{\mathcal{Q}}

\newcommand{\bthetaMC}{\widehat{\btheta}_{MC}}
\newcommand{\bthetaVAR}{\widehat{\btheta}_{VAR}}

\begin{document}

%\begin{frontmatter}
\title{Consistency and Asymptotic Normality of Stochastic Block Models Estimators from Sampled Data}
%\runtitle{Consistency and asymptotic normality of SBM estimators}

%\begin{aug}
%\author{\fnms{MAHENDRA} \snm{MARIADASSOU}\thanksref{a,e1}\ead[label=e1,mark]{mahendra.mariadassou@inrae.fr}}
%\and
%\author{\fnms{TIMOTH\'EE} \snm{TABOUY}\thanksref{b,e2}\ead[label=e2,mark]{timothee.tabouy@agroparistech.fr}}

\author{Mahendra Mariadassou\footnote{mahendra.mariadassou@inrae.fr}  \ \& \ Timoth\'ee Tabouy\footnote{timothee.tabouy@gmail.com}}
\date{$^*$ MaIAGE, INRAE, Universit\'e Paris-Saclay, 78352 Jouy-en-Josas, France \\ $\dag$ UMR MIA-Paris, AgroParisTech, INRA, Universit\'e Paris-Saclay, 75005 Paris, France}
%
%\address[a]{MaIAGE, INRAE, Universit\'e Paris-Saclay, 78352 Jouy-en-Josas, France \\
%\printead{e1}}
%
%\address[b]{UMR MIA-Paris, AgroParisTech, INRA, Universit\'e Paris-Saclay, 75005 Paris,
%France
%\printead{e2}}
%
%\runauthor{M. Mariadassou et al.}
%
%\affiliation{Some University and Another University}
%
%\end{aug}

\maketitle

\begin{abstract}
Statistical analysis of network is an active research area and the literature counts a lot of papers concerned with network models and statistical analysis of networks. However, very few papers deal with missing data in network analysis and we reckon that, in practice, networks are often observed with missing values. In this paper we focus on the Stochastic Block Model with valued edges and consider a MCAR setting by assuming that every dyad (pair of nodes) is sampled identically and independently of the others with probability $\rho > 0$. We prove that maximum likelihood estimators and its variational approximations are consistent and asymptotically normal in the presence of missing data as soon as the sampling probability $\rho$ satisfies $\rho\gg\log(n)/n$. \\

\noindent Stochastic Block Model $\cdot$ Maximum Likelihood $\cdot$ Missing data $\cdot$ Concentration Inequality

\end{abstract} 

%\begin{keyword}
%\kwd{Stochastic Block Model}
%\kwd{Missing data}
%\kwd{Asymptotic normality}
%\kwd{Maximum Likelihood} \kwd{Concentration Inequality}
%\end{keyword}
%
%\end{frontmatter}

%% Intro (TT)
\section{Introduction} \label{sec:introduction}
% flatex input: [introduction.tex]
For the last decade, statistical network analyses has a been a very active research topic and the statistical modeling of networks has found many applications in social sciences and biology for example \cite{Aicher2014},  \cite{Barbillon2015}, \cite{mariadassou2010}, \cite{wasserman1994} and \cite{Zachary1977}.

Many random graphs models have been widely studied, either from a theoretical or an empirical point of view. The first model studied was Erd\H{o}s-R\'enyi model \citep{Erdos1959} which assumes that each pair of nodes (dyad) is connected independently to the others with the same probability.  This model assumes homogeneity of all nodes across the network. In order to alleviate this constraint, many families of models have been introduced. Most are endowed with a latent structure \citep[reviewed in][]{matias2014modeling} to capture heterogeneity across nodes. Among those, the Stochastic Block Model \citep[in short SBM, see][]{frank1982cluster,holland1983stochastic} is one of the oldest and most studied as it is highly flexible and can capture a large variety of structures (affiliation, hub, bipartite and many other). In order to estimate this model, Bayesian approaches were first proposed \citep{snijders1997estimation,Nowicki2001} but have been superseded by variational methods \citep{daudin2008mixture,latouche2012variational}. The former class of approaches are exact but lack the computational efficiency and scalability that the latter offers.

Theoretical guarantees concerning maximum likelihood estimators (in short MLE) and variational estimators (in short VE), based on variational approximations of the likelihood, for the binary SBM estimation are quite difficult to obtain.  In \cite{celisse2012consistency}, consistency of MLE and VE is proven but asymptotic normality requires that the estimators converges at rate at least $n^{-1}$, which is not proven in the paper, although some results were available for some particular cases (affiliation for example). \cite{ambroise2012new} tackles the specific case of affiliation model with equal group proportion and proves the consistency and asymptotic normality of parameter estimates. \cite{bickel2013asymptotic} extends those results to arbitrary binary SBM graphs and improves \cite{celisse2012consistency} by removing the condition on the convergence rate, as it is automatically satisfied by the MLE. Following along the path of \cite{bickel2013asymptotic}, \cite{BraultKeribinMariadassou2017} proved consistency and asymptotic normality of estimators (MLE and VE) to weighted Latent Block Models where the weights distribution belongs to a one-dimensional exponential families. In particular, considering unbounded edge values invalidates several parts of the proofs for binary graphs and requires substantial adaptations and additional results, notably concentration inequalities for sums of unbounded, non-gaussian random variables.   

Some results are also available for the related semi-parametric problem of assignment reconstruction. \cite{mariadassou2015} show that the conditional distribution of the (latent) assignments converge to a degenerate distribution and \cite{BinYu2010} prove  that, when the data are generated according to a SBM model, spectral methods are consistent. \cite{Choi2012StochasticBW} extend those results to settings where the density of the graph goes to $0$ as $\Omega(\log^\alpha(n)/n)$ (for $\alpha$ large enough) and/or the number of groups goes to $+\infty$ as $\sqrt{n}$. 
\citet{Chatterjee2015} proves  also strong results for reconstruction of large matrices with noisy entries estimation and partial observation of the dyads, by means of a universal singular value thresholding (USVT). In the special case of binary SBM with $k$ groups, he achieves a reconstruction error rate of order $\sqrt{k/n}$ as soon as the fraction of observed dyads is at least $\Omega(\log^\alpha(n)/n)$ for (for $\alpha$ large enough). Since USVT replaces missing dyads with $0$s, it naturally achieves the same limiting rate as the sparse setting. Finally, \cite{wang2017} and \cite{hu2016} also show that model selection for the number $k$ of groups is consistent for dense graphs, they suggest using a penalized likelihood criteria with penalty of the form $\frac{k(k+1)}{2}\log(n) + \lambda n \log(k)$ where $\lambda$ is a tuning parameter.

In this paper we consider a simple setting with fixed number of groups and fixed density but weighted edges and missing values. In most network studies, there is a strong asymmetry between the presence of an edge and its absence: the lack of proof that an edge exists is taken as proof that the edge does not exist and edges with uncertain status are considered as non existent in the graph. This is the strategy adopted in most sparse asymptotic settings where the density of edges goes to $0$ asymptotically \cite{bickel2013asymptotic}. We adopt a different point of view where edges with uncertain status are considered as missing, rather than absent and explicitly accounted for their missing nature. We use the framework of \cite{Rubin1976} and its application to network data, see \cite{Kolaczyk2009} and \cite{Handcock2010}, for parameter inference in presence of missing values and more specifically its applications to SBM \cite{Tabouy2019}. We prove that, in the MCAR setting where each dyad is missing independently and with the same probability, the MLE and variational estimates are still consistent and asymptotically normal. 

The article is organized as follows. We first present the model and missing data theory applied to our context with some examples of sampling designs. We then posit some definitions and discuss the  assumptions required for our results in Section~\ref{sec:statistical-framework}. In Section~\ref{sec:complete_model} we establish asymptotic normality for the complete-observed model (\textit{i.e.} observed SBM where latent variables are known). Section~\ref{sec:main} is the main result of this paper and states that the observed-likelihood behaves like the complete-observed likelihood (\textit{i.e.} joint likelihood of the observed data and latent variables) close to its maximum. {Consequences for the MLE and variational estimator are in discussed in Section~\ref{sec:MLE}}. The proof is sketched in Section~\ref{sec:proof}.  Comparison to existing results are made and discussed in Section~\ref{sec:discussion}. Technical lemmas and details of the proofs are available in the appendices.

% flatex input end: [introduction.tex]

%% Intro (TT)

%% Notations et Modèle (MM)
\section{Statistical framework} \label{sec:statistical-framework}
% flatex input: [statistical-framework.tex]
\subsection{Notations}

\begin{center}
\begin{tabular}{|l|c|}
 \hline
 Notation & Definition \\
 \hline
 $\node = \{1, \dots, n\}$     & Node set \\
 $i, j$                        & Node index \\
 $\block = \{1, \dots, Q\}$    & Block set \\
 $l, q$                        & Block index \\
 $\dyad = \node \times \node$  & Dyad set \\
 $y_{ij}$ and $r_{ij}$         & Dyad value and observation status\\
 $z_i$ and $\zs_i$             & Test and true block membership \\
 $\bpi$ and $\bpis$            & Test and true block proportions \\
 $\bal$ and $\bals$            & Test and true parameters of dyad value distribution \\
 $\bTheta$                     & Parameter space \\
 \hline
\end{tabular}
\end{center}

\subsection{Stochastic Block Model}

In  SBM,  nodes  from  a  set  $\node  \triangleq  \{1,\dots,n\}$  are
distributed among a set $\block  \triangleq \{1, \dots, Q\}$ of hidden
blocks that  model the latent structure  of the graph.  The block-memberships are encoded by $(z_i,  i\in\node)$  where the $z_i$ are independant random variables with prior
probabilities $\alpha  = ( \alpha_{1}, \dots,  \alpha_{Q})$, such that
$\Pbb(z_i = q) = \alpha_q$,  for all $q\in\block$. The value $\edge_{ij}$ of any dyad $(i, j)$ in $\dyad = \node \times \node$, with $i \neq j$, only depends on the blocks $i$ and $j$ belong to. The variables $(\edge_{ij})$s are thus independent conditionally on the $(z_i)$s:
\begin{equation}
\nonumber
\edge_{ij}    \   |    \   z_i=q,    z_j   =    \ell   \sim^{\text{ind}}
\dens(., \pi_{q\el}), \quad \forall (i,j) \in\dyad, \quad i \neq j , \quad \forall
(q,\ell) \in\block\times\block.
\end{equation}
In the following, $\MA=(\edge_{ij})_{i,j\in\dyad}$ is the $n\times n$ adjacency matrix of the random graph, $\bz = (z_1, \dots, z_n)$ the $n$-vector of the latent blocks.  With a slight abuse of notation, we associate to $z_i$ a binary vector $(z_{i1}, \dots, z_{iQ})$ such that
$z_i = q \Leftrightarrow z_{iq} = 1, z_{i\ell} = 0$, for all
$\ell \neq q$. In this case $\bz$ is a $n \times Q$ matrix.\\

We note the complete parameter set as $\btheta=(\bpi,\bal)\in\bTheta$ where $\bTheta$ stands for the parameter space. When performing inference from data, we note $\bthetas = (\bpis, \bals)$ the true parameter set, \emph{i.e.} the parameter values used to generate the data, and $\bzs$ the true (and usually unobserved) memberships of nodes. For any $\bz$, we also note:
\begin{itemize}
\item $\zsumk = \sum_{\ii} \zik$ the size {of the $\kk^{th}$ community (or block)} for membership $\bz$
\item $\zssumk$ its counterpart for $\bzs$.
\end{itemize}

\subsection{Missing data for SBM}

Regarding SBM inference, a missing value corresponds to a missing entry in the adjacency matrix $\MA$, typically denoted by \texttt{NA}'s. We rely on the $n\times n$ sampling matrix $\br$ to record the missing state of each entry:
\begin{equation}
\label{eq:R}
 (r_{ij}) = \begin{cases}
  1 &   \text{ if $\edge_{ij}$  is observed,}  \\
  0 &   \text{ otherwise.}  \\
\end{cases}
\end{equation}
As  a  shortcut,  we use  $\MAO  =  \{\edge_{ij}  :  r_{ij} =  1\}$  and
$\MAM=\{\edge_{ij} :  r_{ij} =  0 \}$  to respectively denote the \textit{observed} and \textit{missing} dyads.  The \emph{sampling  design} is the description
of the stochastic process that generates  $\br$.  It is assumed that the
network exists before the sampling design acts upon it, which is fully characterized by the conditional distribution $p_\psi(\br|\MA)$, the parameters of which are such   that   $\psi$   and   $\theta$   live   in   a   product   space $\Theta   \times   \Psi$. In this paper we are going to focus on a specific type of missingness, called missing completely at random (MCAR) for which $p_\psi(\br|\MA) = p_\psi(\br)$ and leave aside more complex forms of dependencies such as Missing at random (MAR) and Not missing at random (NMAR). 

We then follow the framework of \citep{Rubin1976} and \cite{Tabouy2019} for  missing data and define the joint probability density function as
\begin{equation}
p_{\theta, \psi}(\MAO, \bz, \br)=\int p_{\theta}(\MAO,\MAM, \bz)p_\psi(\br|\MAO,\MAM, \bz)\textrm{d}\MAM.
\label{eq:likelihood}
\end{equation}

\begin{property}\label{prop:mar} According to Equation
  \eqref{eq:likelihood}, if the sampling design is MCAR, then maximising $p_{\theta, \psi}(\MAO, \bz, \br)$ or $p_{\theta, \psi}(\MAO, \br)$ in $\theta$  is equivalent  to maximising $p_{\theta}(\MAO)$ in $\theta$, this corresponds to the ignorability notion defined in \cite{Rubin1976}.
\label{MAR}
\end{property}

\subsection{Sampling design examples}
\label{sec:sampling}

We present here some examples of sampling designs to illustrate differences between notions of MCAR, MAR and NMAR.

\begin{dof}[Random dyad sampling]
  Each   dyad    $(i,j)   \in\dyad$    has   the    same   probability
  $\Pbb(r_{ij}=1)=\rho$ of being observed, independently of the others. This design is MCAR. 
\end{dof}

\begin{dof}[Random node sampling]
 The \textit{random node sampling} consists in selecting independently with probability $\rho$ a set of nodes and then observing the corresponding rows and columns of matrix $\by$.
\end{dof}

The major point in both examples is that the probability ($\rho$ in \textit{random dyad sampling} and $1-(1-\rho)^2$ in the \textit{random node sampling}) of observing a dyad does not depend on its value.
%\begin{dof}[Snowball sampling]
%  The snowball sampling consists in selecting uniformly a set of nodes,
%  then observing corresponding rows of matrix $\MA$.  It gives a first
%  "wave" of nodes. The second wave is made up of the neighbors of the
%  first. Successive waves can then be obtained. The final set of
%  observed dyads corresponds to all dyads involving at least one of
%  these nodes.
%\end{dof}
In contrast, the following dyad-centered sampling design adapted to binary networks is NMAR since the probability to observe a dyad depends on its value:

\begin{dof}[Double standard sampling]
  Each   dyad    $(i,j)   \in\dyad$  is observed, independently of other dyads, with a probability depending on its value: $\Pbb(r_{ij}=1|y_{ij}=0)=\rho_0$ and $\Pbb(r_{ij}=1|y_{ij}=1)=\rho_1$.
\end{dof}

For non-binary networks, specifying the sampling design is more involved and requires defining the sampling density for every possible value of $y_{ij}$, \emph{e.g.} $(\Pbb(r_{ij}=1|y_{ij}=k))_{k \in \mathbb{N}}$ for Poisson-valued edges. \\

\begin{rem} In this paper, we focused on data sampled according to \textit{random dyad sampling}, which is the simplest case but already yields valuable insights into the differences between the partially and fully sampled settings. \\
As observed above, there are however many other ways to sample a network. In the case of node-centered sampling design, like \textit{random node sampling}, the main difficulty to prove consistency and asymptotic normality is the dependency between the $r_{ij}$ variables. Indeed, in random node sampling, the variable $r_{i_0 j_0}$ depends on all $r_{i j_0}$ and $r_{i_0 j}$ (for all $i, j \in \node$). As a consequence, a different inference strategy is required and many results proved in this paper are not valid under \emph{random node sampling}. NMAR sampling designs raises problem of their own: each design requires its own estimation procedure \citep{Tabouy2019} and therefore its own analysis. For example, parameter estimation under the seemingly simple \emph{double standard sampling} for binary networks is still an open problem: numerical experiments suggest that $\btheta = (\bal, \bpi)$ and $\boldsymbol{\psi} = (\rho_0, \rho_1)$ are jointly identifiable but there is no formal proof. 
\end{rem}

\subsection{Observed-likelihoods}
When the labels are known, the {\em complete-observed log-likelihood} is given by:
\begin{equation}
  \label{eq:log-vraisemblance-complete}
    \Lc(\bz;\btheta) = \log\prob(\byo,\bz;\btheta) = \sum_{i,q} \ziq \log \pii_q +  \sum_{\substack{ i, j, q, \ell \\ i \neq j}} \ziq\zjl\rrij \log \dens (\yij;\piql)
\end{equation}
But the labels are usually unobserved, and the {\em observed log-likelihood} is obtained by integration over all memberships:
\begin{equation}
  \label{eq:log-vraisemblance-observee}
  \LL_o(\btheta)=\log\prob(\byo;\btheta) = \log \left( \sum_{\bz \in \mcZ}\prob(\byo,\bz;\btheta) \right).
\end{equation}

\subsection{Models and Assumptions}\label{sec:assumptions}
We focus here on parametric models where $\dens$ belongs to a regular one-dimension exponential family in canonical form:
\begin{equation} \label{eq:exp-distr}
  \dens(y, \pi) = b(y)\exp(\pi y - \norm(\pi)),
\end{equation}
where $\pi$ belongs to the space  $\mathcal{A}$, so that $\dens(\cdot, \pi)$ is well defined for all $\pi \in \mathcal{A}$. Classical properties of exponential families ensure that $\norm$ is convex, infinitely differentiable on $\mathring{\mathcal{A}}$, that $\normpm$ is well defined on $\normp(\mathring{\mathcal{A}})$. Furthemore, when $\edge_{\pi} \sim \dens(., \pi)$, $\Esp[\edge_\pi] = \normp(\pi)$ and $\Var[\edge_{\pi}] = \normp'(\pi)$. \\

In the following, we recall assuming that missing data are produced according to a \textit{random dyad sampling} with parameter $\rho>0$. \\

Moreover, we make the following assumptions on the parameter space {and the asymptotics of $\rho$} ~:
\begin{enumerate}
\item[$A_0$]: {$\rho$ goes to $0$ but satisfies $\rho \gg \log(n)/n$}
\item[$A_1$]: There exists a positive constant $c$, and a compact interval $C_\pi$ such that 
  \begin{equation*}
    \bTheta \subset [c, 1-c]^{\q} \times C_{\pi}^{\q \times \q} \quad \text{with} \quad C_\pi \subset \mathring{\mathcal{A}}.
  \end{equation*}
\item[$A_2$]: The true parameter $\bthetas = (\bpis, \bals)$ lies in the interior of $\bTheta$.
\item[$A_3$]: The map $\pi \mapsto \dens(\cdot, \pi)$ is injective.
\item[$A_4$]:  The coordinates of $\bals \normp(\bpis)$, where $\normp$ is applied component-wise, are pairwise distinct.
\end{enumerate}

The previous assumptions are standard. {Assumption~$A_0$ ensures that the fraction of observed dyad is not too small}. Assumption~$A_1$ ensures that the group proportions are bounded away from $0$ and $1$ so that no group disappears when $\n$ goes to infinity. It also ensures that $\pi$ is bounded away from the boundaries of the $\mathcal{A}$. This is essential for the subexponential properties of Propositions~\ref{prop:subexp_variables} and \ref{prop:subexp_mixture}. $A_2$ is in line with standard assumptions in parametric statistics. $A_3$ is necessary for identifiability purposes: the model is trivially not identifiable if the map $\pi \mapsto \dens(., \pi)$ is not injective. $A_4$ ensures identifiability of SBM parameters under \textit{random dyad sampling}. Note that, combined with $A_3$, it implies that all columns and all rows of $\bals$ are distincts and therefore that no two groups have the connectivity profile. In the following, we  consider the number of blocks $\mathcal{Q}$ to be known.

\subsection{Identifiability}
\label{sec:identifiability}
Since $\br$ is independant on $\by$, the identifiability of SBM with emission law in the one-dimension exponential family under \textit{random dyad sampling} can be stated in two steps. First the sampling parameter $\rho$ and secondly the SBM parameters $\bthetas = (\bpis, \bals)$ given $\rho$. 

\begin{proposition}
\label{prop:identfiabilitySampParam}
The sampling parameter $\rho > 0$ of \textit{random dyad sampling}
is identifiable w.r.t. the sampling distribution.
\end{proposition}
\begin{proof}
See \cite{Tabouy2019}. The proof does not depend on $\by$ being binary but also holds for $\by$ distributed as in Eq.~\eqref{eq:exp-distr}. 
\end{proof}

\begin{proposition}\label{prop:identifiabilitySBM}
  Let $n\geq 2Q$ and assume that for any $1\leq q \leq Q$, $\rho>0$,
  $\als_q >0$ and that the coordinates of $\bpis \normp(\bals)$, where $\normp$ is applied component-wise, are pairwise   distinct.  Then, under \textit{random dyad sampling}, SBM
  parameters are identifiable w.r.t. the distribution of the observed
  part of the SBM up to label switching.
\end{proposition}

\begin{proof}
The proof is nearly identical to the one written in \cite{Tabouy2019} and inspired by \cite{celisse2012consistency} for the binary SBM under \textit{random dyad sampling}. However, substituting $\mathbb{E}[y_{ij}|z_i=q]$ to $s_q$ in the proof ensures that $\bpis$ is identifiable. Finally, the fact that $(\psi^\prime)^{-1}$ is a one-to-one map ensures that $\bals$ is identifiable.
\end{proof}
Note that asymptotically, the assumption $n \geq 2Q$ is always satisfied since $Q$ is fixed and $n$ grows to infinity. 

\subsection{Subexponential variables}
\label{sec:def_subexp}

\begin{rem}
Since we restricted $\pi$ in a bounded subset of $\mathring{\mathcal{A}}$, the variance of $y_{\pi}$ is bounded away from $0$ and $+\infty$. We note
\begin{equation}
  \label{eq:condition-variance}
  \vmax^2 = \sup_{\pi \in C_\pi } \Var(y_\pi) < +\infty \quad \text{and} \quad \vmin^2 =  \inf_{\pi \in C_\pi} \Var(y_\pi) > 0.
\end{equation}
Similarly, since $\pi$ belongs to a bounded subset of a open interval, there exists a constant $\kappa>0$, such that $[\pi - \neighborsize, \pi + \neighborsize] \subset \mathring{\mathcal{A}}$ uniformly over all $\pi \in C_{\pi}$
\end{rem}

\begin{proposition} \label{prop:subexp_variables}
With the previous notations, if $\pi \in C_\pi$ and $y_{\pi} \sim \dens(., \pi)$, then $y_\pi$ is subexponential with parameters $(\vmax^2, \neighborsize^{-1})$.
\end{proposition}

\begin{proposition} \label{prop:subexp_mixture}
Considering $x=y_\pi r_{ij} + \lambda r_{ij}$ (we recall that $r_{ij} \sim \mathcal{B}(\rho)$), with $r_{ij}$ independant of $y_\pi$ and $\lambda \in \mathbb{R}$ bounded. There are non-negative numbers $\nu$ and $b$ such that $x$ is subexponential with parameters $(\nu^2, b^{-1})$.
\end{proposition}
\begin{proof}
These results derive directly from theorem \ref{thm:subexp_eq} (statement 2.).
\end{proof}

\subsection{Symmetry}
\label{sec:definitions}

We now introduce the concepts of assignments and parameter symmetries, that must be accounted for when studying the asymptotic properties of the MLE. Complications stemming from symmetries are related to but no equivalent to the problem of label-switching in mixture models. 

\begin{dof}[permutation]
  \label{def:permutation}
  Let $s$ be a permutation on $\{1,\dots,\Q\}$. If $\boldsymbol{A}$ is a matrix with $\Q$ columns and $n$ rows, we define $\boldsymbol{A}^s$ as the matrix obtained by permuting the columns of $\boldsymbol{A}$ according to $s$, \emph{i.e.} for any row $\ii$ and column $\kk$ of $\boldsymbol{A}$, ${A}^s_{\ii \kk} = A_{\ii s(\kk)}$. If $\boldsymbol{C}$ is a matrix with $\Q$ rows and $\Q$ columns, $\boldsymbol{C}^{s}$ is defined similarly:
  \begin{equation*}
    \boldsymbol{A}^s = \left( A_{\ii s^{}(\kk)} \right)_{\ii,\kk} \quad \boldsymbol{C}^{s} = \left( C_{s^{}(\kk) s^{}(\el)} \right)_{\kk,\el}
  \end{equation*}
\end{dof}

\begin{dof}[equivalence]
  \label{def:equivalence}
  We define the following equivalence relationships:
  \begin{itemize}
  \item Two assignments $\bz$ and $\bz'$ are \emph{equivalent}, noted $\sim$, if they are equal up to label permutation, \emph{i.e.} there exists a permutation $s$ such that $\bz' = \bz^s$.
  \item Two parameters $\btheta$ and $\btheta'$ are \emph{equivalent}, noted $\sim$, if they are equal up to label permutation, \emph{i.e.} there exists a permutation $s$ such that $(\bpi^s,\bal^{s}) = (\bpi', \bal')$. 
  \item  $(\btheta, \bz)$ and $(\btheta', \bz')$ are \emph{equivalent}, noted $\sim$, if they are equal up to label permutation on $\bal$ and $\bz$, \emph{i.e.} there exists a permutation $s$ such that $(\bal^{s}, \bz^s) = (\bal', \bz')$. This is \emph{label-switching}.
  \end{itemize}
\end{dof}

\begin{dof}[symmetry]
  \label{def:symmetry}
  We say that the parameter $\btheta$ \emph{exhibits symmetry for the permutation} $s$ if
  \begin{equation*}
    (\bpi^s, \bal^{s}) = (\bpi, \bal).
  \end{equation*}
  $\btheta$ \emph{exhibits symmetry} if it exhibits symmetry for any non trivial permutations $s$. Finally the set of permutations for which $\btheta$ exhibits symmetry is noted $\Symmetric(\btheta)$.
\end{dof}

\begin{rem}
  The set of parameters that exhibit symmetry is a manifold of null Lebesgue measure in $\bTheta$.  The notion of symmetry allows us to deal with a notion of non-identifiability of the class labels that is subtler than and different from label switching. More precisely
  \begin{eqnarray*}
   \textrm{Label switching is when :}  \quad p(\byo,\bz,\btheta) &=& p(\byo,\bz^s,\btheta^s), \ \theta \neq \theta^s \ \forall s \\
   \textrm{Symmetry is when :}  \quad p(\byo,\bz,\btheta) &=& p(\byo,\bz^s,\btheta), \ \forall s \in \textrm{Sym}(\theta)
  \end{eqnarray*}
  In particular, in label-switching, $\bz$ and $\bz^s$ have the same likelihood but under equivalent yet different parameters $\btheta$s. In contrast, in the presence of symmetry, $\bz$ and $\bz^s$ have exactly the same likelihood under $\btheta$. This implies in particular that the posterior $p(\bz|\byo,\btheta)$ can not concentrate on a single assignment. This is instrumental for Proposition~\ref{prop:equivalent-configurations-profile-likelihood}.
\end{rem}

{
\begin{example} In this example we illustrate what $\textrm{Sym}(\theta)$ and its cardinal can be in a simple case. Consider a network with $n$ nodes, 
\begin{equation*}
\alpha=(1/6, 1/6, 2/3), \ \text{and} \ \ \pi=\left(\begin{array}{ccc}0 & 0.7 & 0.2 \\0.7 & 0 & 0.2 \\0.2 & 0.2 & 0.2\end{array}\right).
\end{equation*}
As a consequences the two following assignments
\begin{eqnarray*}
z_1&=&(\underbrace{1,\cdots,1}_{n_1},\underbrace{2,\cdots,2}_{n_2},\underbrace{3,\cdots,3}_{n_3}) \\ z_2&=&(\underbrace{2,\cdots,2}_{n_2},\underbrace{1,\cdots,1}_{n_1},\underbrace{3,\cdots,3}_{n_3})
\end{eqnarray*}
belongs to $\textrm{Sym}(\theta) = \{ Id, [1, 2]\}$. Indeed they are the only assignments belongings to $\textrm{Sym}(\theta)$, so, in this particular case $\#\textrm{Sym}(\theta)=2$.
\end{example}
}
The issue of symmetry forces us to use a notion of distance between assignment that is invariant to label permutation.

\begin{dof}[distance]
  \label{def:equivalence-distance}
  We define the following distance, up to equivalence, between configurations $\bz$ and $\bzs$:
    \begin{equation*}
    \|\bz - \bzs\|_{0, \sim} = \inf_{\bz' \sim \bz} \|\bz' - \bzs\|_0
    \end{equation*}
    where, for all matrix $\bz$, we use the Hamming norm $\left\|\cdot\right\|_{0}$ defined by
\[\left\|\bz\right\|_{0}=\frac{1}{2}\sum_{\ii,\kk}{\mathds{1}{\left\{\zik\neq0\right\}}}.\]
\end{dof}
  
\begin{dof}[Set of local assignments]
  \label{prop:small-deviations-profile-likelihood}
  We note $S(\bzs, r)$ the set of configurations that have a representative (for $\sim$) within relative radius $r$ of $\bzs$:
  \begin{equation*}
    S(\bzs, r) = \left\{ \bz : \|\bz - \bzs\|_{0, \sim} \leq r \n \right\}
  \end{equation*}
\end{dof}

%The last equivalence relationship is not concerned with $\bal$. It is useful when dealing with the conditional likelihood $\prob(\by| \bz; \btheta)$ which does not depend on $\bal$ :  in fact, if $(\btheta, \bz) \sim (\btheta', \bz')$, then for all $\by$, we have $\prob(\by| \bz; \btheta) = \prob(\by| \bz'; \btheta')$. 
%Note also that $\bz \sim \bzs$  if and only if the confusion matrix $\RQbz$  is equivalent to a diagonal matrix. 

\subsection{Other definitions}

We finally introduce a few useful notions that will be instrumental in the proofs. The first is ``regular'' assignments, for which each group has ``enough'' nodes:
\begin{dof}[$c$-regular assignments]
  \label{def:regular}
  Let $\bz \in \mcZ$. For any $c > 0$, we say that $\bz$ is c-\emph{regular} if
  \begin{equation}
    \label{eq:regular-configuration}
    \min_{\kk} \zsumk \geq {c\n} .
  \end{equation}
\end{dof}

Class distinctness $\delta(\bal)$ captures the differences between groups: lower values of $\delta(\bal)$ means that at least two classes have very similar connectivity profiles. $\delta(\bal)$ is intrisically linked to the convergence rate of several estimates. 
\begin{dof}[class distinctness]
  \label{def:group-distinctness}
  For $\btheta = (\bpi, \bal) \in \bTheta$. We define:

  \begin{equation*}
    \delta(\bal) = \min_{\kk, \kp} \max_{\el} \KL(\alkl, \al_{\kp\el})
  \end{equation*}
  
  with $\KL(\pi,\pi') = \Esp_{\pi}[\log(\dens(Y, \pi)/\dens(Y, \pi'))]=\normp(\pi) (\pi - \pi') + \norm(\pi') - \norm(\pi)$ the Kullback divergence between $\dens(., \pi)$ and $\dens(., \pi')$, when $\dens$ comes from an exponential family.
\end{dof}
\begin{rem}
Since all $\bal$ have distinct rows and columns, $\delta(\bal) > 0$.
\end{rem}

Finally, the confusion matrix allows to compare groups between assignments:
\begin{dof}[confusion matrix]
  \label{def:confusion}
  For given assignments $\bz$ and $\bzs$, we define the \emph{confusion matrix} between $\bz$ and $\bzs$, noted $\RQbz$, as follows:
  \begin{equation}
    \label{eq:confusion-matrix}
    \RQbz_{q q^{\prime}} = \frac{1}{\n} \sum_{\ii} \z^\vrai_{\ii q} \z_{\ii q^{\prime}}
  \end{equation}
\end{dof}

\begin{dof}
For more conciseness, we define
\begin{equation}
 \Sal  = (\Salkl)_{\kk\el} = \left( \normp(\alskl) \right)_{\kk\el}
\end{equation}
\end{dof}

% flatex input end: [statistical-framework.tex]

%% Notations et Modèle (MM)

%% Résultats asymptotiques sur le modèle complet (TT)
\section{Complete-observed Model} \label{sec:complete_model}
% flatex input: [complete_model.tex]

{Hereafter and in the rest of the text, we use the term "complete" to say that true assignments $\bzs$ are known, and "observed" to say that only some dyads are observed.}
In the following we study the asymptotic properties of the complete-observed data model. 
\begin{proposition}
\label{prop:isolated_nodes}
Under random dyad sampling, defining $N_i = \sum_{j}r_{ij}$ and $\Omega_{0,n} = \cap_{i=1}^{n} \{ N_i \geqslant 1 \}$ the set of nodes with at least one dyad observed. Then 
$$\mathbb{P}\left(\underset{n\to+\infty}{\lim}\Omega_{0,n}\right) = 1.$$
\label{prop:Omega}
\end{proposition}

\begin{proof}
This proposition is a direct consequence of Borel-Cantelli's theorem. Details are available in appendix \ref{appendix:technical_results}.
\end{proof}

\begin{rem}
This result shows that, with high probability, the network has no unobserved node. In the remainder, we work conditionnally on $\Omega_{0,n}$.
\end{rem}

Let $\widehat{\btheta}_{c}=\left(\hbpi,\hbal\right)$ be the MLE of $\btheta$ in the complete-observed data model. Simple manipulations of Equation~\eqref{eq:log-vraisemblance-complete} yield:

\begin{equation}
  \label{eq:mle-complete-likelihood}
  \begin{aligned}
    \hpi_{q}  = \pizq = \frac{\zsumq}{\n}&& \\
    \hyqlz = \frac{\sum_{\ii\neq \jj} \yij \rrij \ziq \zjl}{\sum_{\ii \neq \jj} \rrij \ziq \zjl} & \quad
    \hal_{q\el}  = \pizql = \normpm \left(  \hyqlz \right)&
  \end{aligned}
\end{equation}

\begin{proposition}
  \label{prop:mle-asymptotic-normality}
  Let $\Sigma_{\bpis} = \Diag(\bpis) - \bpis\left(\bpis\right)\transpose$.Then $\Sigma_{\bpis}$ is semi-definite positive, of rank $\mathcal{Q}-1$, and $\widehat{\bpi}$ is asymptotically normal:
  \begin{equation}
    \label{eq:mle-proportion-asymptotic-normality}
    \sqrt{\n}\left( \hat{\bpi}\left(\bzs\right) - \bpis \right) \xrightarrow[\n \to \infty]{\mathcal{D}} \mathcal{N}(0, \Sigma_{\bpis})
  \end{equation}
Similarly, let $V(\bals)$ be the matrix defined by $[V(\bals)]_{q \el} = 1/\normp'(\pi^\vrai_{q\el})$ and\\  $\Sigma_{\bals} = \rho^{-1}\Diag^{-1}(\bpis) V(\bals) \Diag^{-1}(\bpis)$. Then the estimates $\hat{\pi}_{q\ell}(\bzs)$ are independent and asymptotically Gaussian with limit distribution:
  \begin{equation}
    \label{eq:mle-parameters-asymptotic-normality}
    \sqrt{\n(\n-1)}\;(\hal_{q\el}\left(\bzs\right) - \pi^\vrai_{q\el}) \xrightarrow[\n \to \infty]{\mathcal{D}}  \mathcal{N}(0, \Sigma_{\bals, q \el})\;\;\; \hbox{for all  } q, \el
  \end{equation}
\end{proposition}

\begin{proof}
The proof is postponed to appendix \ref{appendix:technical_results}. The first part is a direct application of central limit theorem for i.i.d. variables and the second part relies on a variant of the central limit theorem for random sums of random variables. 
\end{proof}

\begin{rem} 
The main differences with \cite{bickel2013asymptotic} are (i) the scaling of $\Sigma_{\bals}$ as $\rho^{-1}$ and (ii) the need for a central limit theorem for random sums of random variables, as the sums involved in \eqref{eq:mle-complete-likelihood} are over a random number of terms. 
\end{rem}

\begin{proposition}[Local asymptotic normality]\label{prop:LocalAsymp}
Let $\Lcs$ be the complete likelihood function defined on $\bTheta$ by $\Lcs\left(\bpi,\bal\right)=\log\prob\left(\by^o,\bzs;\btheta\right)$. For any $s$ and $u$ in a compact set, we have:
\begin{eqnarray*}
\Lcs\left(\bpis+\frac{s}{\sqrt{\n}},\bals+\frac{u}{\sqrt{\n(\n-1)}}\right)
&=&\Lcs\left(\bthetas\right) + s\transpose\bY_{\bpis}+ \text{Tr}(u\transpose\bY_{\bals})\\
& -&\left(\frac{1}{2}s\transpose\Sig_{\bpis}s+\frac{1}{2} \text{Tr}\left((u \odot u)\transpose\Sig_{\bals}\right)\right)\\
&+&o_{P}(1)\\
\end{eqnarray*}
where $\odot$ denote the Hadamard product of two matrices (element-wise product) and $\Sig_{\bpis}$ and $\Sig_{\bals}$ are defined in Proposition~\ref{prop:mle-asymptotic-normality}. $\bY_{\bpis}$ is asymptotically Gaussian with zero mean and variance matrix $\Sig_{\bpis}$. $\bY_{\bals}$ is a random matrix with independent entries that are asymptotically gaussian zero mean and variance $\Sig_{\bals}$.
\end{proposition}

\begin{proof}
This result is based on a Taylor expansion of $\Lcs$ in a neighborhood of $(\bpis,\bals)$. Details are available in appendix \ref{appendix:technical_results}.
\end{proof}

% flatex input end: [complete_model.tex]

%% Résultats asymptotiques sur le modèle complet (TT)

%% Equivalence entre vraisemblances des modèles complet et observé (MM)
\section{Main Result} \label{sec:main}
% flatex input: [main_result.tex]
Our main result compares the observed likelihood ratio $\prob(\by^o; \btheta)/\prob(\by^o; \bthetas)$ with the complete observed likelihood  $\prob(\by^o, \bzs; \btheta')/\prob(\by^o, \bzs; \bthetas)$ to show that they have the same argmax. To ease the comparison, we work only on the high probablity set $\Om_1$ of $c/2$-regular configurations, \emph{i.e.} that have $\Omega(n)$ nodes in each group as defined in Section~\ref{sec:statistical-framework}, 

\begin{proposition}
  \label{cor:prob-regular-configurations-star}
  Define $\mcZ_1$ as the subset of $\mcZ$ made of $c/2$-regular assignments, with $c$ defined in assumption $H_1$. Note $\Om_1$ the event $\{ \bzs
  \in \mcZ_1 \}$, then:
  \begin{equation*}
    \Prob_{\bthetas}\left( \bar{\Om}_1 \right) \leq \q \exp\left( -\frac{\n c^2}{2}\right).
  \end{equation*}
\end{proposition}

\begin{proof}
This proposition is a consequence of Hoeffding's inequality. See appendix \ref{appendix:technical_results} for more details.
\end{proof}

%%%%%%%%%%%%%%%%%%%%%%
We can now state our main result: 

\begin{thm}[complete-observed]
  \label{thm:observed-akin-to-complete-general}
  Assume that $A_1$ to $A_4$ with random-dyad sampling hold for the Stochastic Block Model of known order with $n\times n$ observations coming from an univariate exponential family  and define  $\# \Symmetric(\btheta)$ as the set of permutation $s$ for which $\btheta=(\bpi,\bal)$ exhibits symmetry. Then, for $n$ tending to infinity and $\rho \gg \log (n)/n$, the observed likelihood ratio behaves like the complete likelihood ratio, up to a bounded multiplicative factor:
  \begin{equation*}
    \frac{\prob(\byo; \btheta)}{\prob(\byo; \bthetas)} = \frac{\# \Symmetric(\btheta)}{\# \Symmetric(\bthetas)} \max_{\btheta' \sim \btheta} \frac{\prob(\byo, \bzs; \btheta')}{\prob(\byo, \bzs; \bthetas)}\left(1 + \smallO_P(1)\right) + \smallO_P(1)
  \end{equation*}
  where the $\smallO_P$ is uniform over all $\btheta \in \bTheta$.
\end{thm}
The maximum over all $\btheta'$ that are equivalent to $\btheta$ stems from the fact that because of label-switching, $\btheta$ is only identifiable up to its $\sim$-equivalence class from the observed likelihood, whereas it is completely identifiable from the complete likelihood. The multiplicative factor arises from the fact that equivalent assignments have exactly the same complete likelihood and contribute equally to the observed likelihood. \\

\begin{rem}
 This result is very similar to the one of \cite{BraultKeribinMariadassou2017} and corrects an error in the main result of \cite{bickel2013asymptotic}: the missing terms $\# \Symmetric(\btheta)$ and $\# \Symmetric(\bthetas)$.
\end{rem}

\begin{corollaire}
  \label{cor:observed-akin-to-complete-simple-case}
  If $\bTheta$ contains only parameters with no symmetry:
  \begin{equation*}
    \frac{\prob(\byo; \btheta)}{\prob\left(\byo; \btheta^{\vrai}\right)} = \max_{\btheta' \sim \btheta} \frac{\prob(\byo, \bzs; \btheta')}{\prob(\byo, \bzs; \bthetas)}\left(1 + \smallO_P(1)\right) + \smallO_P(1)
  \end{equation*}
where the $\smallO_P$ is uniform over all $\bTheta$.
\end{corollaire}

% flatex input end: [main_result.tex]

%% Equivalence entre vraisemblances des modèles complet et observé (MM)

%% Conséquences pour les estimateurs du MV et variationnels 
\section{Variational and Maximum Likelihood Estimates}
% flatex input: [other_estimates.tex]
\label{sec:MLE}
This section is devoted to the asymptotic of the MLE and the VE in the incomplete data model as a consequence of the main result \ref{thm:observed-akin-to-complete-general}. Note that, with high probability, both estimators have no symmetry since the set $\{ \btheta : \# \Symmetric(\btheta) > 1\} $ is a manifold of null Lebesque's mesure in $\bTheta$ and thus $\Prob_{\bthetas}(\# \Symmetric(\hat{\btheta}) > 1) \to 0$. 
\subsection{ML estimator}

The asymptotic behavior of the maximum likelihood estimator in the incomplete data model is a direct consequence of Theorem~\ref{thm:observed-akin-to-complete-general} and Proposition~\ref{prop:LocalAsymp}. 

\begin{corollaire}[Asymptotic behavior of $\bthetaEMV$]\label{cor:behaviorEMV} Denote $\bthetaEMV$ the maximum likelihood estimator and use the notations of Proposition~\ref{prop:mle-asymptotic-normality}. There exist permutations $s$ of $\{1, \dots, \Q\}$ such that
\begin{eqnarray*}
\hat{\bpi}\left(\bzs\right)-\bpiEMV^{s}&=&o_{P}\left(\n^{-1/2}\right),\\
\hat{\bal}\left(\bzs\right)-\balEMV^{s}&=&o_{P}\left(\n^{-1}\right).\\
\end{eqnarray*}
% If $\# \Symmetric(\btheta)\neq1$,  $\bthetaEMV$ is still consistent:  there exist permutations $s$ of $\{1, \dots, \Q\}$ such that
% \begin{eqnarray*}
% \hat{\bpi}\left(\bzs\right)-\bpiEMV^{s}&=&o_{P}\left(1\right),\\
% \hat{\bal}\left(\bzs\right)-\balEMV^{s}&=&o_{P}\left(1\right). \\
% \end{eqnarray*}
\end{corollaire}
Hence, the maximum likelihood estimator for the SBM under random-dyad sampling condition is consistent and asymptotically normal, with the same behavior as the maximum likelihood estimator in the complete data model. The proof is postponed to appendix \ref{annexe:cor:behaviorEMV}. 

\subsection{Variational estimator}

Due to the complex dependency structure of the observations, the maximum likelihood estimator of the SBM is not numerically tractable, even  with the \textit{Expectation Maximisation} algorithm. In practice, a variational approximation is often used \citep[see][]{daudin2008mixture}: for any joint distribution $\setQ\in\mcQ$ on $\mcZ$ a lower bound of $\mcL(\btheta)$ is given by
\begin{eqnarray*}
\Jvar\left(\setQ,\btheta\right)&=&\mcL(\btheta)-KL\left(\setQ,\prob\left(. ; \btheta,\by^o\right)\right)\\
                            &=&\Esp_{\setQ}\left[\Lc\left(\bz;\btheta\right)\right]+\mcH\left(\setQ\right).
\end{eqnarray*}
where $\mcH\left(\setQ\right)=-\Esp_{\setQ}[\log(\setQ)]$.
Choosing $\mcQ$  to be the set of product distributions, such that for all $\bz$
\[\setQ\left(\bz\right)=\prod_{\ii,\kk}\setQ\left(\zik=1\right)^{\zik} \]
allows us to obtain tractable expressions of $\Jvar\left(\setQ,\btheta\right)$. The variational estimate $\bthetavar$ of $\btheta$ is defined as
\[\bthetavar\in\underset{\btheta\in\bTheta}{\arg\!\max}\;\underset{\setQ\in\mcQ}{\max}\;\Jvar\left(\setQ,\btheta\right).\]

The following corollary states that $\bthetavar$ has the same asymptotic properties as $\bthetaEMV$ and $\widehat{\btheta}_{MC}$, in particular is consistent and asymptotically normal. 
\begin{corollaire}[Variational estimate]\label{cor:Variational}
Under the assumptions of Theorem~\ref{thm:observed-akin-to-complete-general}, there exist permutations $s$ of $\{1, \dots, \Q\}$ such that
\begin{eqnarray*}
\hat{\bpi}\left(\bzs\right)-\bpivar^{s}&=&o_{P}\left(\n^{-1/2}\right),\\
\hat{\bal}\left(\bzs\right)-\balvar^{s}&=&o_{P}\left(\n^{-1}\right).
\end{eqnarray*}
\end{corollaire}

The proof is very similar to the proof of Corollary~\ref{cor:behaviorEMV} and postponed to appendix \ref{annexe:cor:behaviorEMV}. \\

\section{Proof Sketch} \label{sec:proof}
% flatex input: [proof_sketch.tex]
The proof of theorem relies on deviations of the log-likelihood ratios from their expectations. We first define those quantities. 

\subsection{log-likelihood ratios}
\label{sec:cond-and-prof-likelihood} 

\begin{dof}
\label{def:conditional-profile-likelihood}
 We define the conditional log-likelihood ratio $\Fn$ and its expectation $\G$ as:
\begin{equation}
  \label{eq:conditional-likelihood}
    \Fn(\btheta, \bz) = \log \frac{\prob(\byo | \bz;\btheta)}{\prob(\byo | \bzs;\bthetas)} \quad \text{and} \quad 
    \G(\btheta, \bz) = \Esp_{\bthetas} \left[ \left. \Fn(\btheta, \bz) \right| \bzs  \right]
\end{equation}
We also define the profile ratio $\Lamb$ and its counterpart $\Lambtilde$ as:
\begin{equation}
  \label{eq:profile-likelihood}
    \Lamb(\bz)      = \max_{\btheta} \Fn(\btheta, \bz)
    \quad \text{and} \quad
    \Lambtilde(\bz) = \max_{\btheta} \G(\btheta, \bz).
\end{equation}
\end{dof}

The following decomposition of $\prob(\byo; \btheta)$ highlights the importance of $\Fn(\btheta, \bz )$:

\begin{equation*}
\prob(\byo; \btheta) 
= \sum_{ (\bz)} \prob(\byo, \bz; \btheta) 
= \prob(\byo| \bzs; \bthetas) \sum_{(\bz)}   \prob(\bz; \btheta) \exp(\Fn(\btheta, \bz) ).
\end{equation*}

Since $\Fn(\btheta, \bz) \leq \Lamb(\bz)$, the profile ratio is useful to remove the dependency on $\btheta$ and reduce the study to a series of problems depending only on $\bz$. The following propositions show that $\Lambtilde$ and $\G$ are constrats which are maximum (in expectation) at the true parameter value (up to group relabeling) and have negative curvature at those points. This allows us to prove that, asymptotically, only one (or a few) $z$ contribute to the above sum.

\begin{proposition}
\label{prop:expectation-ychap}
 Conditionally on  $\bzs$, we have
\begin{equation}
\label{eq:expectation-ychap}
  \begin{aligned}
    \barykl(\bz) &:= \Esp_{\bthetas}[\hyklz | \bzs] = \frac{\left[ \RQbz\transpose \Sal \RQbz \right]_{\kk\el}}{\pizk\pizl}
  \end{aligned}
\end{equation} 
with $\barykl(\bz)=0$ for $\bz$ such that $\pizk=0$ or $\pizl=0$ \emph{i.e.} no dyad observed in class $(q,l)$.
\end{proposition}

\begin{rem}
Note the absence of the random variable $\br$ in $\barykl(\bz)$, which is integrated out in the expectation $\Esp_{\bthetas}$.
\end{rem}

\begin{proposition}[maximum of $\G$ and $\Lambtilde$ in $\btheta$]
\label{prop:profile-likelihood}
The functions $\Fn(\btheta, \bz)$ and $\G(\btheta, \bz)$ are maximum respectively in $\bal$ for  $\hbal(\bz)$ and $\bar{\bal}(\bz)$ defined by:
\begin{equation*}
  \hal(\bz)_{\kk\el} = \normpm ( \hyklz ) \quad \text{and} \quad \bar{\al}(\bz)_{\kk\el} = \normpm ( \barykl(\bz) )
\end{equation*}
so that
\begin{equation*}
    \Lamb(\bz) = \Fn(\hbal(\bz), \bz) 
    \quad \text{and} \quad
    \Lambtilde(\bz) = \G(\bar{\bal}(\bz), \bz).
\end{equation*}
\end{proposition}

% Note that although $\bar{y}_{\kk\el} = \mathbb{E}_{\bthetas}\left[\left.\hat{y}_{\kk\el}\right|\bzs\right]$, in general %$\bar{\al}_{\kk\el} = \normpm(\barxkl) \neq \normpm(\hxkl) = \hat{\al}_{\kk\el}$
% $\bar{\al}_{\kk\el} \neq \mathbb{E}_{\bthetas}\left[\left.\hat{\al}_{\kk\el}\right|\bzs\right]$ by non linearity of $\normpm$. Nevertheless, $\normpm$ is Lipschitz over compact subsets of $\normp(\mathring{\mathcal{A}})$ and therefore, with high probability, $|\bar{\al}_{\kk\el} - \hat{\al}_{\kk\el}|$ and $|\hykl - \barykl|$ are of the same order of magnitude. 

\begin{proposition}[Local upperbound for  $\Lambtilde$]
\label{prop:profile-likelihood-derivative}
Conditionally upon $\Om_1$, there exists a positive constant $C$ such that for all $\bz \in S(\bzs, C)$:
\begin{equation}
\label{eq:conditional-likelihood-separability}
  \Lambtilde(\bz) \leq - c \rho \n\frac{3\delta(\bals)}{4} \|\bz - \bzs\|_{0, \sim}
\end{equation}
\end{proposition}

\begin{proposition}[maximum of $\G$ and $\Lambtilde$ in $(\btheta,\bz)$]
  \label{prop:maximum-conditional-likelihood}
  $\G$ can be written:
  \begin{equation}
    \label{eq:conditional-likelihood-second-form}
    \G(\btheta, \bz) = - \rho\n^2 \sum_{\kk,\kp} \sum_{\el,\lp} \RQbz_{\kk,\kp} \RQbz_{\el, \lp} \KL(\al^\vrai_{\kk\el}, \al_{\kp\lp}) \leq 0.
  \end{equation}
  Conditionally on the set $\Om_1$ of regular assignments and for $n> 2/c$,
\begin{itemize}
\item[(i)] $\G$ is maximized at $(\bals, \bzs)$ and its equivalence class and $\G(\bals, \bzs) = 0$.
\item[(ii)] $\Lambtilde$ is maximized at $\bzs$ and its equivalence class and $\Lambtilde(\bzs)=0$.
\item[(iii)] The maximum of $\Lambtilde$ (and hence the maximum of $\G$) is well separated.
\end{itemize}
\end{proposition}

Proofs of Propositions~\ref{prop:expectation-ychap}, \ref{prop:profile-likelihood}, \ref{prop:profile-likelihood-derivative} and \ref{prop:maximum-conditional-likelihood} are postponed to Appendix \ref{appendix:main-results}.

\subsection{High level view of the proof}

The proof proceeds by splitting $\prob(\byo;\btheta)$ as a sum over three types of configurations that partition $\mcZ$ and studying the asymptotic behavior of  $\Fn$ and on each type:

\begin{enumerate}
\item  \emph{global control}: for $\bz$ such that $\Lambtilde(\bz) = \Omega(-\n^2)$, Proposition \ref{prop:conditional-likelihood-convergence} proves a large deviation behavior and shows that $\Fn = -\Omega_P(\n^2)$. In turn, those assignments contribute a $\smallO_{P}$ of $\prob(\byo, \bzs; \bthetas))$ to the sum (Proposition \ref{prop:large-deviations-profile-likelihood}). 

\item \emph{local control}: a small deviation result (Proposition~\ref{prop:profile-likelihood-convergence-local}) is needed  to show that the combined contribution of assignments close to but not equivalent to $\bzs$ is also a $\smallO_{P}$ of $\prob(\byo, \bzs; \bthetas)$ (Proposition \ref{prop:small-deviations-profile-likelihood}). 
\item \emph{equivalent assignments}: Proposition~\ref{prop:equivalent-configurations-profile-likelihood} examines which of the remaining assignments, all equivalent to $\bzs$, contribute to the sum.
\end{enumerate}
These results are presented in next section \ref{sec:technicalpropositions} and their proofs postponed to Appendix \ref{appendix:main-results}. They are then put together in section \ref{sec:proofbigth} to prove our main result. The remainder of the section is devoted to the asymptotics of the ML and variational estimators as a consequence of the main result.

\subsection{Different asymptotic behaviors}\label{sec:technicalpropositions}

% We begin with a large deviations inequality for configurations $\bz$ far from $\bzs$ and leverage it to prove that far away configurations make a small contribution to $\prob(\byo; \btheta)$. 

\subsubsection{Global Control}

\begin{proposition}[large deviations of $\Fn$]
  \label{prop:conditional-likelihood-convergence}
  Let $\Diam(\bTheta) = \sup_{\btheta, \btheta'} \|\btheta - \btheta' \|_{\infty}$. For all $\vareps_{\n} < \nu b$ and $\n$ large enough that $2\sqrt{2n^2}\epsilon_{n} \geq Q^2$
%   \begin{multline}
%     \label{eq:conditional-likelihood-convergence}
%     \Delta_{\n}^1(\vareps_{\n})\\ = \Prob\left( \sup_{\btheta, \bz} \left\{ \Fn(\btheta, \bz) - \Lambtilde(\bz) \right\} \geq \nu \n^2 \Diam(\bTheta) 2\sqrt{2}\vareps_{\n}\left[1 + \frac{\q^2}{2\sqrt{2}n\vareps_{\n}} \right] \right) \\ 
%     \leq \q^{\n} \exp\left( - \frac{\n^2\vareps_{\n}^2} {2}\right)
%   \end{multline}
\begin{equation}
 \sup_{\btheta, \bz} \left\{ \Fn(\btheta, \bz) - \Lambtilde(\bz) \right\} = \bigO_{p}(n^2\epsilon_n)
\end{equation}

\end{proposition}

\begin{proposition}[contribution of global assignments]
  \label{prop:large-deviations-profile-likelihood}
  Choose $t_n$ decreasing to $0$ slowly enough that $\frac{\rho\n\tn}{\sqrt{\log(n)}} \to +\infty$. Then conditionally on $\Omega_1$ and for $\n$ large enough that $2\sqrt{2\n^2} \epsilon_n \geq \q^2$, we have:
  \begin{equation*}
    \sup_{\btheta \in \bTheta} \sum_{\bz \notin S(\bzs, t_n)} \prob(\bz, \byo; \btheta) = \smallO_P( \prob(\bzs, \byo; \bthetas) )
  \end{equation*}
\end{proposition}

\subsubsection{Local Control}

% Proposition~\ref{prop:conditional-likelihood-convergence} gives deviations of order $\bigO_{P}(\n)$, which are only useful for $\bz$ such that $\G$ and $\Lambtilde$ are large compared to $\n$. For $\bz$ close to $\bzs$, we need tighter concentration inequalities, of order $\smallO_P(-\n)$, as follows:

\begin{proposition}[small deviations $\Fn$]
  \label{prop:profile-likelihood-convergence-local}
  Conditionally on $\Om_1$,
  \begin{equation}
    \label{eq:profile-likelihood-convergence-local}
    \sup_{\bz \nsim \bzs}\frac{\Lamb(\bz) - 
\Lambtilde(\bzs)}{\n \|\bz - \bzs\|_{0, \sim}} = \smallO_P(1)
  \end{equation}
\end{proposition}

The next proposition uses  Propositions~\ref{prop:profile-likelihood-convergence-local} and \ref{prop:maximum-conditional-likelihood} to show that the combined contribution to the observed likelihood of assignments close to $\bzs$  is also a $\smallO_P$ of  $\prob(\bzs, \by^o; \bthetas)$:

\begin{proposition}[contribution of local assignments]
  \label{prop:small-deviations-profile-likelihood}
  With the previous notations and $C$ the positive constant defined in Proposition~\ref{prop:profile-likelihood-derivative}:
  \begin{equation*}
    \sup_{\btheta \in \bTheta} \sum_{\substack{\bz \in S(\bzs, C) \\ \bz \nsim \bzs}} \prob(\bz, \by^o; \btheta) = \smallO_P( \prob(\bzs, \by^o; \bthetas) )
  \end{equation*}
\end{proposition}

\subsubsection{Equivalent assignments}
It remains to study the contribution of equivalent assignments.
\begin{proposition}[contribution of equivalent assignments]
  \label{prop:equivalent-configurations-profile-likelihood}
  For all $\btheta \in \bTheta$, we have 
  \begin{equation*}
    \sum_{\bz \sim \bzs} \frac{\prob(\byo, \bz; \btheta)}{\prob(\byo, \bzs; \bthetas)} = \# \Symmetric(\btheta) \max_{\btheta' \sim \btheta} \frac{\prob(\byo, \bzs; \btheta')}{\prob(\byo, \bzs; \bthetas)} (1 + \smallO_P(1))
  \end{equation*}   
  where the $\smallO_P$ is uniform in $\btheta$. 
\end{proposition}

\subsection{Proof of the main result}\label{sec:proofbigth}

\begin{proof}
We work conditionally on $\Omega_1$. Choose $\bzs
\in \mcZ_1$ and a sequence $\tn$ decreasing to $0$ but satisfying $\rho\n\tn/\sqrt{\log(\n)} \to +\infty$. According to Proposition~\ref{prop:large-deviations-profile-likelihood}, 
\begin{equation*}
    \sup_{\btheta \in \bTheta} \sum_{\bz \notin S(\bzs, \tn)} \prob(\bz, \by^o; \btheta) = \smallO_P( \prob(\bzs, \by^o; \bthetas) )
\end{equation*}
Since $\tn$ decreases to $0$, it gets smaller than $C$ (used in proposition~\ref{prop:small-deviations-profile-likelihood}) for $\n$ large enough. As this point, Proposition~\ref{prop:small-deviations-profile-likelihood} ensures that:
\begin{equation*}
    \sup_{\btheta \in \bTheta} \sum_{\substack{\bz \in S(\bzs, \tn) \\ \bz \nsim \bzs}} \prob(\bz, \by^o; \btheta) = \smallO_P( \prob(\bzs, \by^o; \bthetas) )
\end{equation*}
And therefore the observed likelihood ratio reduces as:
\begin{align*}
    \frac{\prob(\by^o; \btheta)}{\prob(\by^o; \bthetas)} & = \frac{\displaystyle \sum_{\bz \sim \bzs} \prob(\by^o, \bz; \btheta) + \sum_{\bz \nsim \bzs} \prob(\by^o, \bz; \btheta)}{\displaystyle \sum_{\bz \sim \bzs} \prob(\by^o, \bz; \bthetas) + \sum_{\bz \nsim \bzs} \prob(\by^o, \bz; \bthetas)} \\
    & = \frac{\displaystyle \sum_{\bz \sim \bzs} \prob(\by^o, \bz; \btheta) + \prob(\by^o; \bzs, \bthetas) \smallO_P(1)}{\displaystyle \sum_{\bz \sim \bzs} \prob(\by^o, \bz; \bthetas) + \prob(\by^o; \bzs, \bthetas) \smallO_P(1)} \\
\end{align*}
And Proposition~\ref{prop:equivalent-configurations-profile-likelihood} allows us to conclude
\begin{equation*}
    \frac{\prob(\by^o; \btheta)}{\prob(\by^o; \bthetas)} = \frac{\# \Symmetric(\btheta)}{\# \Symmetric(\bthetas)} \max_{\btheta' \sim \btheta} \frac{\prob(\by^o, \bzs; \btheta')}{\prob(\by^o, \bzs; \bthetas)}(1 + \smallO_P(1)) + \smallO_P(1).
\end{equation*}

\end{proof}

\section{Discussion} \label{sec:discussion}

Close examination of the different proofs, especially of Prop.~\ref{prop:small-deviations-profile-likelihood}, reveals that the quantities driving convergence of the estimates are $\rho \n \delta(\bal^\star)$, which must go to $+\infty$ with $n$ to ensure validity of Prop.~\ref{prop:large-deviations-profile-likelihood}, and $\rho \n \tn \delta(\bal^\star)$, which must be larger than $\sqrt{\log(n)}$ while $\tn \to 0$, to ensure validity of Prop.~\ref{prop:small-deviations-profile-likelihood}. Both conditions are met as soon as $\rho \gg \log(n)/n$, allowing for a large fraction of missing edges. Note that this limiting rate for missingness is the same as the one found for graph density in sparse settings to achieve consistency and local asymptotic normality of $\btheta$ \citep{bickel2013asymptotic}. It's also the same as the one found by \cite{Chatterjee2015} for the structured matrix reconstruction problem. Note also that in the fixed $\rho$ setting, both MLE and VE are consistent and asymptotically normal but the cost of missingness is an expected blow up of the asymptotic variance matrix by a factor of $\rho^{-1}$.

The proof follows along the line of \citep{bickel2013asymptotic} but differs in some significant ways. First, since the number of observed dyads is random, we must rely on variants of the central limit theorem that hold for random sums of random variables. Second, the move from the binary to unbounded dyads invalidates a counting argument used in \citep{bickel2013asymptotic} and requires different concentration inequalities. We leverage the facts that random variables with distribution in natural exponential families are subexponential and that the subexponential property is preserved by summation and multiplication to derive Bernstein-type inequality. Finally, we add the missing terms $\# \Symmetric(\btheta)$ which have little impact for the corollaries but are required for the rigorous statement of the main result.

%In this paper, we focused on data sampled according to \textit{random dyad sampling}. However, as described in section~\ref{sec:sampling}, there are many other ways to sample a network. In the case of node-centered sampling design, like \textit{random node sampling}, the main difficulty to prove consistency and asymptotic normality is the dependency between the $r_{ij}$ variables. Indeed, in random node sampling, the variable $r_{i_0 j_0}$ depends on all $r_{i j_0}$ and $r_{i_0 j}$ (for all $i, j \in \node$). As a consequence, many results proved in this paper are not valid under \emph{random node sampling}. NMAR sampling designs raises problem of their own: each design requires its own estimation procedure \citep{Tabouy2019} and therefore its own analysis. For example, even parameter estimation under the \emph{double standard sampling} for binary networks mentioned in section~\ref{sec:sampling} is still an unsolved problem: numerical experiments suggest that $\btheta = (\bal, \bpi)$ and $\boldsymbol{\psi} = (\rho_0, \rho_1)$ are jointly identifiable but there is no formal proof. 

\section{Acknowledgment} 

The authors thank Pierre Barbillon (INRA-MIA, AgroParisTech), Julien Chiquet (INRA-MIA, AgroParisTech), 
St\'ephane Robin (INRA-MIA, AgroParisTech) and James Ridgway (CFM) for their helpful remarks and suggestions.

This work is supported by two public grants overseen by the French
National research Agency (ANR): first as part of the « Investissement
d’Avenir » program, through the « IDI 2017 » project funded by the
IDEX Paris-Saclay, ANR-11-IDEX-0003-02, and second by the « EcoNet »
project.

% flatex input end: [proof_sketch.tex]

%% Grandes lignes de la preuve (TT sur les , MM)

%% modifier les notations pour les simplifier et les rendre homogènes    
\newpage
%% Appendices et preuves
\appendix

\section{Technical results} \label{sec:appendices}
% flatex input: [appendix_proof-technical-results.tex]
\label{appendix:technical_results}

\subsection{Proof of proposition~\ref{prop:isolated_nodes}}
\begin{proof}
Noticing that $N_i \sim \text{Bin}(n-1, \rho)$, then  $\mathbb{P}(N_i \geqslant 1) = 1 - (1-{\rho})^{n-1}$. As a consequence
$\mathbb{P}(\overline{\Omega_{0,n}}) \leqslant  \sum_{i} \mathbb{P}(N_i = 0) = n(1-{\rho})^{n-1} \underset{n\to+\infty}{\longrightarrow} 0$, and 
$\mathbb{P}(\Omega_{0,n}) \underset{n\to+\infty}{\longrightarrow} 1$. Then
$\mathbb{P}(\limsup (\overline{\Omega_{0,n}})) = 0$ by Borel-Cantelli theorem (because $\sum_{n} \mathbb{P}(\overline{\Omega_{0,n}})$ converge), and as
$\overline{\limsup \overline{\Omega_{0,n}}} = \overline{\bigcap_{n \geqslant 0} \bigcup_{q \geqslant n} \overline{\Omega_{0,n}}} = \bigcup_{n \geqslant 0} \bigcap_{q \geqslant n} {\Omega_{0,n}} =
\liminf {\Omega_{0,n}}$, the result follow.
\end{proof}

\subsection{Technical lemma~\ref{lem:convdenom}}

\begin{lemme}
$$U_{n} = \frac{1}{\n(\n-1)}\sum_{\ii \neq \jj} \rrij \ziq \zjl \xrightarrow[\n \to +\infty]{\mathbb{P}} \rho \pii_{\kk}\pii_{l}$$
\label{lem:convdenom}
\end{lemme}

\begin{proof}
Noticing that $\mathbb{E}[\rrij \ziq \zjl] = \rho \pii_{\kk}\pii_{l}$ and defining $q_{i,j}^{q, \el}=\rrij \ziq \zjl - \rho \pii_{\kk}\pii_{l}$. By Hoeffding decomposition for U-statistics (see \cite{Hoeffding1948})
\begin{equation}
 \label{eq:ustat}
 U_{n}^\prime = \frac{1}{\n(\n-1)}\sum_{\ii \neq \jj} (\rrij \ziq \zjl - \rho \pii_{\kk}\pii_{l}) = \frac{1}{n!}\sum_{\sigma \in \mathfrak{S}_{n}}\frac{1}{\lfloor\frac{n}{2}\rfloor}\sum_{i=1}^{\lfloor\frac{n}{2}\rfloor}q_{\sigma(i), \sigma(i+\lfloor\frac{n}{2}\rfloor)}^{q, \el},
\end{equation}
where for each permutation $\sigma \in \mathfrak{S}$, $\sum_{i=1}^{\lfloor\frac{n}{2}\rfloor}q_{\sigma(i), \sigma(i+\lfloor\frac{n}{2}\rfloor)}^{q, \el}$ is a sum of independant r.v. Then, for $\gamma > 0$ by Jensen's inequality and Hoeffding's lemma about bounded r.v.
\begin{eqnarray*}
\mathbb{E}\left[ \exp(\gamma U_{n}^\prime) \right] &\leq& \frac{1}{n!}\sum_{\sigma \in \mathfrak{S}_{n}}\mathbb{E}\exp\left(\frac{\gamma}{\lfloor\frac{n}{2}\rfloor}\sum_{i=1}^{\lfloor\frac{n}{2}\rfloor}q_{\sigma(i), \sigma(i+\lfloor\frac{n}{2}\rfloor)}^{q, \el}\right) \\
&\leq& \exp\left(\frac{\gamma^2}{8\lfloor\frac{n}{2}\rfloor}\right).
\end{eqnarray*}
Finally, using the same proof than Hoeffding's inequality allows us to conclude.
\end{proof}

\subsection{Proof of proposition~\ref{prop:mle-asymptotic-normality}}
\begin{proof}
Since $\hat{\bpi}\left(\bzs\right) = \left(\hpi_1\left(\bzs\right), \dots, \hpi_\g\left(\bzs\right)\right)$ is the sample mean of $\n$ i.i.d. multinomial random variables with parameters $1$ and $\bpis$, a simple application of the central limit theorem (CLT) gives:
\begin{equation*}
  \Sigma_{\bpis, \kk \kp} =
  \begin{cases}
    \pii^{\vrai}_{\kk}(1 - \pii^{\vrai}_{\kk}) & \text{if} \quad \kk = \kp \\
    -\pii^{\vrai}_{\kk} \pii^{\vrai}_{\kp} & \text{if} \quad \kk \neq \kp \\
  \end{cases}
\end{equation*}
which proves Equation~\eqref{eq:mle-proportion-asymptotic-normality}
where $\Sigma_{\bpis}$ is
semi-definite positive of rank $\mathcal{Q} - 1$.

Similarly, $\normp\left(\hal_{\kk\el}\left(\bzs\right)\right)$ is
the average of $\sum_{\ii \neq \jj} \rrij \ziq^\star \zjl^\star$
i.i.d. random variables with mean
$\normp\left(\al^\vrai_{\kk\el}\right)$ and variance
$\normp'\left(\al^\vrai_{\kk\el}\right)$. $\sum_{\ii \neq\jj} \rrij \ziq^\star \zjl^\star$
is itself random but thanks to lemma \ref{lem:convdenom} : 
\mbox{$\frac{1}{\n(\n-1)}\sum_{\ii \neq \jj} r_{ij} \ziq^\star \zjl^\star \xrightarrow[\n \to +\infty]{\mathbb{P}} \rho \pii^\vrai_{\kk}\pii^\vrai_{l}$}. Therefore, by Slutsky's lemma and
the CLT for random sums of random variables \cite{Shanthikumar1984}, we have:
\begin{eqnarray*}
&&\sqrt{{\n(\n-1)}\rho\pii^\star_{\kk}\pii^\star_{\el}} \left(\normp\left(\hal_{\kk\el}\left(\bzs\right)\right) - \normp(\al^\vrai_{\kk\el}) \right)  \\
&& = \sqrt{{\n(\n-1)}\rho\pii^\star_{\kk}\pii^\star_{\el}} \left( \frac{\sum_{\ii \neq  \jj} y_{ij} r_{ij} \zik^{\vrai} \zjl^{\vrai}}{\sum_{\ii \neq  \jj} r_{ij} \ziq^\star \zjl^\star} - \normp(\al^\vrai_{\kk\el}) \right) \\
&& \xrightarrow[\n \to +\infty]{\mathcal{D}} \mcN\left(0, \normp'(\al^\vrai_{\kk\el})\right)
\end{eqnarray*}
The differentiability of $\normpm$ and the delta method then gives:
\begin{equation*}
  \sqrt{{\n(\n-1)}} \left(\hal_{\kk\el}\left(\bzs\right) - \al^\vrai_{\kk\el} \right) \xrightarrow[\n \to +\infty]{\mathcal{D}} \mcN\left(0, \frac{1}{\rho\pii^\star_{\kk}\pii^\star_{\el} \normp'(\al^\vrai_{\kk\el})}\right)
\end{equation*}
and the independence results from the independence of $\hal_{\kk\el}\left(\bzs\right)$ and
$\hal_{\kp\lp}\left(\bzs\right)$ as soon as $\kk \neq \kp$ or $\el
\neq \lp$, as they involve different sets of i.i.d. variables.
\end{proof}

\subsection{Proof of proposition~\ref{prop:LocalAsymp}}

\begin{proof}
By Taylor expansion,
\begin{eqnarray*}
& &\!\!\!\!\!\!\!\!\!\!\Lcs\left(\bpis+\frac{s}{\sqrt{\n}},\bals+\frac{u}{\sqrt{{\n(\n-1)}}}\right)\\
&=&\Lcs\left(\bthetas\right)+\frac{1}{\sqrt{\n}}s\transpose\nabla{\Lcs}_{\bpi}\left(\bthetas\right) +\frac{1}{\sqrt{{\n(\n-1)}}}\text{Tr}\left(u\transpose\nabla{\Lcs}_{\bal}\left(\bthetas\right)\right)\\
&&\quad+\frac{1}{\n}s\transpose\bH_{\bpi}\left(\bthetas\right)s+\frac{1}{{\n(\n-1)}}\text{Tr}\left((u \odot u)\transpose\bH_{\bal}\left(\bthetas\right)\right)+o_{P}(1)\\
\end{eqnarray*}
where $\nabla{\Lcs}_{\bpi}\left(\bthetas\right)$ and $\nabla{\Lcs}_{\bal}\left(\bthetas\right)$ denote the respective components of the gradient of $\Lcs$ evaluated at $\bthetas$ and $\bH_{\bpi}$ and $\bH_{\bal}$ denote the conditional hessian of $\Lcs$ evaluated at $\bthetas$. By inspection, $\bH_{\bpi}/\n$ and $\bH_{\bal}/{(\n(\n-1))}$ converge in probability to constant matrices $\Sigma_{\alpha}, \Sigma_{\pi}$ and the random vectors $\nabla{\Lcs}_{\bpi}\left(\bthetas\right)/\sqrt{\n}$ and $\nabla{\Lcs}_{\bal}\left(\bthetas\right)/\sqrt{{\n(\n-1)}}$ converge in distribution by central limit theorem.
\end{proof}

\subsection{Proof of proposition~\ref{cor:prob-regular-configurations-star}}

\begin{proof}
In regular configurations, each group  has $\Om(\n)$ members, where $u_\n=\Om(\n)$ if there exists two constant $a, b>0$ such that for $\n$ enough large $a\n \leq u_\n \leq b\n$. $c/2$-regular assignments, with $c$ defined in Assumption $H_1$, have high  $\Prob_{\bthetas}$-probability in the space of all assignments, uniformly over all $\bthetas \in \bTheta$. 

Each $\zsumk$ is a sum of $\n$ i.i.d Bernoulli r.v. with parameter $\pik \geq \pii_{\min} \geq c$. A simple Hoeffding bound shows that
\begin{equation*}
  \Prob_{\bthetas}\left( \zsumk \leq \n \frac{c}{2} \right)
  \leq
  \Prob_{\bthetas}\left( \zsumk \leq \n \frac{\pik}{2} \right)
  \leq
  \exp\left( - 2\n\left(\frac{\pik}{2}\right)^2 \right)
  \leq
  \exp\left( - \frac{\n c^2}{2} \right)
\end{equation*}
taking a union bound over $\q$ values of $\kk$ leads to Proposition \ref{cor:prob-regular-configurations-star}.
\end{proof}

% flatex input end: [appendix_proof-technical-results.tex]

%% Appendices et preuves
\section{Main Results}
% flatex input: [appendix_proof-main-results.tex]
\label{appendix:main-results}

\subsection{Proof of proposition~\ref{prop:expectation-ychap})} 

\begin{proof}
First of all we will prove equation \ref{eq:expectation-ychap},
\begin{eqnarray*}
 \barykl(\bz) &=& \Esp_{\bthetas}\left[\frac{\sum_{i\neq j}\ziq\zjl\rrij\yij}{\sum_{i\neq j}\ziq\zjl\rrij} \middle| \bzs \right] \\
 &=& \Esp_{\bthetas}\left[\Esp_{\bthetas}\left[\frac{\sum_{i\neq j}\ziq\zjl\rrij\yij}{\sum_{i\neq j}\ziq\zjl\rrij} \middle| R, \bzs\right] \middle| \bzs\right] \\
 &=& \Esp_{\bthetas}\left[\frac{\sum_{i\neq j}\ziq\zjl\rrij S_{Z_{i} Z_{j}}^{\star}}{\sum_{i\neq j}\ziq\zjl\rrij} \middle| \bzs\right],\\
 \end{eqnarray*}
where $Z_{i} = q \Leftrightarrow \ziq = 1$. Noticing that the $(i,j)$ for which $\ziq\zjl=0$ does not contributes in any of the two terms of the ratio. The calculus of this expectation is then equivalent to calculate an expectation of the general form $\Esp_{\bthetas}\left[ \frac{\sum_{i=1}^{n} a_i R_i}{\sum_{i=1}^{n} R_i} \right]$, $(a_i)_{i\in\{1,..,n\}} \in \mathbb{R}^n$ and $T_i \stackrel{iid}{\sim} \mathcal{B}(\rho)$.
\begin{lemme}
$$\Esp_{\bthetas}\left[ \frac{\sum_{i=1}^{n} a_i T_i}{\sum_{i=1}^{n} T_i} \right] = \frac{\sum_{i=1}^{n}a_i}{n}.$$
\label{lem:expectation}
\end{lemme}
\begin{proof}
Define $N=\sum_{i=1}^{n} T_i$ and noticing that $\Esp[T_i | N=k]=\frac{k}{n}$. Conditionally to $N \geq 1$ \begin{eqnarray*}
\Esp\left[ \frac{\sum_{i=1}^{n} a_i T_i}{\sum_{i=1}^{n} T_i} \right] &=& \Esp\left[\Esp\left[ \frac{\sum_{i=1}^{n} a_i T_i}{N} \middle| N\right]\right] \\
&=& \frac{\sum_{i=1}^{n}a_i}{n}. 
\end{eqnarray*}
\end{proof}

Now, applying lemma \ref{lem:expectation} with $N_{q\el}^o(z) = \sum_{i\neq j}\ziq\zjl\rrij$ leads to
\begin{equation*}
\Esp_{\bthetas}\left[\frac{\sum_{i\neq j}\ziq\zjl\rrij S_{Z_{i} Z_{j}}^{\star}}{\sum_{i\neq j}\ziq\zjl\rrij}\middle| \bzs, N_{q\el}^o(z) \geq 1 \right] = \frac{\left[ \RQbz\transpose \Sal \RQbz \right]_{\kk\el}}{\pizk\pizl} \1_{N_{q\el}^o(z) \geq 1}.
\end{equation*}
Finally, $\Esp_{\bthetas}[\hyklz | \bzs, N_{q\el}^o(z)=0]$ can be arbitrarily defined at the same value than $\Esp_{\bthetas}[\hyklz | \bzs, N_{q\el}^o(z)\geq1]$ which conclued the proof.
\end{proof}

\subsection{Proof of proposition~\ref{prop:profile-likelihood}} 
\begin{proof}
Defining
    $\nu(y, \al)  = y \al - \norm (\al)$. For $\y$ fixed, $\nu(y, \al)$ is maximized at $\al = \normpm(y)$. Manipulations yield
\begin{align*}
  &\Fn(\btheta, \bz)= \log \prob(\byo; \bz,\btheta) - \log {\prob(\byo; \bzs,\bthetas)} \\
%   & = \n\dd \left[ \sum_{\kk} \sum_{\el} \pizk \rhowl \left\{\hxklzw \alkl - \norm(\alkl) \right\}  - \sum_{\kk} \sum_{\el} \pizsk \rhowsl \left\{\hxkl(\bzs, \bws) \alskl - \norm(\alskl) \right\} \right] \\
  & = \left[ \sum_{\kk} \sum_{\el} N_{q\el}^o(z) \nu(\hyklz, \alkl)  - \sum_{\kk} \sum_{\el} N_{q\el}^o(z^\star)\nu(\hykl(\bzs), \alskl) \right]
\end{align*}
which is maximized at $\alkl = \normpm(\hyklz)$. Similarly with $N_{q\el}(z) = \sum_{i \neq j}\ziq\zjl$,
\begin{align*}
 & \G(\btheta, \bz)  = \Esp_{\bthetas} [ \log \prob(\byo; \bz,\btheta) - \log \prob(\byo; \bzs,\bthetas) | \bzs ] \\
  & = \rho\left[ \sum_{\kk} \sum_{\el} N_{q\el}(z) \nu(\baryklz, \alkl)  - \sum_{\kk} \sum_{\el} N_{q\el}(z^\star) \nu(\normp(\alskl), \alskl) \right]
\end{align*}
is maximized at $\alkl = \normpm(\baryklz)$.
\end{proof}

\subsection{Proof of Proposition~\ref{prop:maximum-conditional-likelihood}  (maximum of $\G$ and $\Lambtilde$)} 

\begin{proof}
We condition on $\bzs$ and prove Equation~\eqref{eq:conditional-likelihood-second-form}:
\begin{align*}
  \G(\btheta, \bz) & = \Esp_{\bthetas} \left[ \log \left. \frac{\prob(\byo; \bz,\btheta)}{\prob(\byo; \bzs,\bthetas)} \right| \bzs \right] \\
  & = \sum_{\ii} \sum_{\jj} \sum_{\kk,\kp}\ \sum_{\el,\lp} \Esp_{\bthetas}\left[ y_{ij}  (\al_{\kp\lp} - \alskl) - (\norm(\al_{\kp\lp}) - \norm(\alskl)) \right] \rho \z^\vrai_{\ii\kk} \z_{\ii\kp} \z^\vrai_{\jj\el} \z_{\jj\lp} \\
  & = \n^2 \rho \sum_{\kk, \kp} \sum_{\el, \lp}  \RQbz_{\kk,\kp} \RQbz_{\el, \lp} \left[ \normp(\alskl) (\al_{\kp\lp} - \alskl) +  \norm(\alskl) - \norm(\al_{\kp\lp}) \right] \\
  & = - \n^2 \rho \sum_{\kk,\kp} \sum_{\el,\lp} \RQbz_{\kk,\kp} \RQbz_{\el, \lp} \KL(\al^\vrai_{\kk\el}, \al_{\kp\lp})
\end{align*}

If $\bzs$ is regular, and for $n > 2/c$, all the rows of $\RQbz$ have
at least one positive element and we can apply Lemma 3.2 of \cite{bickel2013asymptotic} to characterize the maximum for
$\G$.

The maximality  of $\Lambtilde(\bzs)$ results from the fact that $\Lambtilde(\bz) = \G(\bar{\bal}(\bz), \bz)$ where
$\bar{\bal}(\bz)$ is a particular value of $\bal$, $\Lambtilde$ is
immediately maximum at $\bz \sim \bzs$, and for those, we have $\bar{\bal}(\bz) \sim \bals$.

The separation and local behavior of $G$ around $\bzs$ is a direct consequence of the proposition \ref{prop:profile-likelihood-derivative}.
\end{proof}

\subsection{Proof of Proposition~\ref{prop:profile-likelihood-derivative} (Local upper bound for  $\Lambtilde$)}

\begin{proof}

We work conditionally on $\bzs$. The principle of the proof relies on the extension of $\Lambtilde$ to a continuous subspace of $\mathcal{M}_\mathcal{Q}([0, 1])$, in which the confusion matrix is naturally embedded. The regularity assumption allows us to work on a subspace that is bounded away from the borders of $\mathcal{M}_\mathcal{Q}([0, 1])$. The proof then proceeds by computing the gradient of $\Lambtilde$ at and around its argmax and using those gradients to control the local behavior of $\Lambtilde$ around its argmax. The local behavior allows us in turn to show that $\Lambtilde$ is well-separated.

 Note that $\Lambtilde$ only depends on $\bz$ through $\RQbz$. We can therefore extend it to matrix $U \in \mathcal{U}_c$ where $\mathcal{U}$ is the subset of matrices $\mathcal{M}_{\mathcal{Q}}([0, 1])$ with each row sum higher than $c/2$. 
\[
  \Lambtilde(U)  = - \rho\n^2 \sum_{\kk,\kp} \sum_{\el,\lp} U_{\kk\kp} U_{\el\lp} \KL\left(\alskl, \baral_{\kp\lp} \right)
  \] where
\[ \baralkl = \baralkl(U) = \normpm \left( \frac{\left[U\transpose \Sal U \right]_{\kk\el}}{\left[U\transpose \mathbf{1} U \right]_{\kk\el}} \right) 
\]
and $\mathbf{1}$ is the $\mathcal{Q} \times \mathcal{Q}$ matrix filled with $1$. Confusion matrix $\RQbz$ satisfy $\RQbz \un = \bpi(\bzs)$, with $\un=(1,\ldots, 1)\transpose$ a vector only containing $1$ values, and are obviously in $\mathcal{U}_c$ as soon as $\bzs$ is $c/2$ regular.

The maps $f_{\kk,\kk^\prime,\el,\el^\prime}: (U) \mapsto KL(\alskl, \baralkl(U))$ are twice differentiable with second derivatives bounded over $\mathcal{U}_c$ and therefore so is $\Lambtilde(U)$. Tedious but straightforward computations show that the derivative of $\Lambtilde$ at $D_{\pii} \coloneqq \Diag(\bpi(\bzs))$ is: 

\begin{equation*}
A_{\kk\kp}(\bzs)  \coloneqq \frac{-1}{n^2} \frac{\partial \Lambtilde}{\partial U_{\kk\kp}}(D_{\pii}) = 2\rho  \sum_{\el} \alpha_{\ell}(\bzs) \KL\left(\alskl, \als_{\kp\el} \right)\\
\end{equation*}
% Symmetric case :
% \begin{equation*}
% A_{\kk\kp}(\bws)  \coloneqq \frac{-1}{n^2} \frac{\partial \Lambtilde}{\partial U_{\kk\kp}}(D_{\pii}) = \rho  \sum_{\el} \alpha_{\ell}(\bzs) \left( \KL\left(\alskl, \als_{\kp\el} \right) + \KL\left(\als_{\el q} , \als_{\el\kp} \right) \right)\\
% \end{equation*}
$A(\bzs)$ is the matrix-derivative of $-\Lambtilde/\n^2$ at $D_{\pii}$. Since $\bzs$ is $c/2$-regular and by definition of $\delta(\bals)$, $A(\bzs)_{\kk\kp} \geq c\rho\delta(\bals)$ if $\kk \neq \kp$ and $A(\bzs)_{\kk\kk} = 0$ for all $\kk$. By boundedness of the second derivative, there exists $C > 0$ such that for all $D_{\pii}$ and all $H \in B(D_{\pii}, C)$, we have:
\begin{align*}
\frac{-1}{n^2} \frac{\partial \Lambtilde}{\partial U_{\kk\kp}}(H) \begin{cases} \geq \rho \frac{7c\delta(\bals)}{8} \text { if } \kk \neq \kp \\ \leq \rho \frac{c\delta(\bals)}{8} \text { if } \kk = \kp \end{cases}
\end{align*}
Choose $U$ in $\mathcal{U}_{c} \cap B(D_{\pii}, C)$ satisfying $U\un = \bpi(\bzs)$. $U - D_{\pii}$ have nonnegative off diagonal coefficients and negative diagonal coefficients. Furthermore, the coefficients of $U, D_{\pii}$ sum up to $1$ and $\Trace(D_{\pii}) = 1$. By Taylor expansion, there exists $H$ also in $\mathcal{U}_c \cap B(D_{\pii}, C)$ such that
\begin{multline*}
\frac{-1}{n^2} \Lambtilde\left( U \right) =  \frac{-1}{n^2} \Lambtilde\left( D_{\pii} \right) + \Trace\left((U - D_{\pii}) \frac{-1}{n^2} \frac{\partial \Lambtilde}{\partial U}(H) \right) \\
\geq \rho\frac{c\delta(\bals)}{8} [ 7 \sum_{\kk \neq \kp} (U - D_{\pii})_{\kk\kp}  - \sum_{\kk} (U - D_{\pii})_{\kk\kk} \\
= c\rho\frac{3\delta(\bals)}{4}(1 - \Trace(U))
\end{multline*}
To conclude the proof, assume without loss of generality that $\bz \in S(\bzs, C)$ achieves the $\|.\|_{0,\sim}$ norm (i.e. it is the closest to $\bzs$ in its representative class). Then $U = \RQbz$ is in $(\mathcal{U}_c \cap B(D_{\pii}, C)$ and satisfy $U\un = \bpi(\bzs)$. We just need to note $\n(1 - \Trace(\RQbz)) = \| \bz - \bzs\|_{0,\sim}$ to end the proof.

\end{proof}

\subsection{Proof of Proposition~\ref{prop:conditional-likelihood-convergence} (global convergence $\Fn$)}

\begin{proof}
Conditionally upon $\bzs$,
\begin{eqnarray*}
   \Fn(\btheta, \bz) - \Lambtilde(\bz) &\leq& \Fn(\btheta, \bz) - \G(\btheta, \bz) \\
  &=& \sum_{\ii} \sum_{\jj} (\al_{\zi\zj} - \al^\vrai_{\zi^\vrai \zj^\vrai}) \left(y_{ij}r_{ij} - \normp(\al^\vrai_{\zi^\vrai \zj^\vrai})\rho \right) \\
  &&  + \sum_{\ii} \sum_{\jj} (\psi(\al_{\zi\zj}) - \psi(\al^\vrai_{\zi^\vrai \zj^\vrai})) \left(r_{ij} - \rho \right) \\
  &=& \sum_{\kk \kp} \sum_{\el \lp} \left(\al_{\kp\lp} - \al^\vrai_{\kk \el} \right) W_{\kk \kp \el \lp} \\
  &\leq& \sup_{\substack{\Gamma \in \R^{ \q^2 \times \q^2} \\ \| \Gamma \|_{\infty} \leq \Diam(\bTheta)}} \sum_{\kk \kp} \sum_{\el \lp} \Gamma_{\kk \kp \el \lp} W_{\kk \kp \el \lp} \coloneqq Z
\end{eqnarray*}
uniformly in $\btheta$, where the $W_{\kk \kp \el \lp}$ are independent and by Taylor expansion defined by:
\begin{equation*}
  W_{\kk \kp \el \lp} = \sum_{\ii} \sum_{\jj} \z^\vrai_{\ii \kk} \z^\vrai_{\jj \el} \z_{\ii, \kp} \z_{\jj \lp}\left(y_{ij}r_{ij} - \normp(\al^\vrai_{\kk\el})\rho - (r_{ij}-\rho)C_{\kk \kp \el \lp} \right), \quad C_{\kk \kp \el \lp} \in \psi^\prime(\Theta)
\end{equation*}
is the sum of $\n^2 \RQbz_{\kk\kp} \RQbz_{\el\lp}$ sub-exponential variables with parameters $(\nu^2, 1/b)$ and is therefore itself sub-exponential with parameters $(\n^2 \RQbz_{\kk\kp} \RQbz_{\el\lp} \nu^2, 1/b)$. According to Proposition B.3 of \cite{BraultKeribinMariadassou2017} %~\ref{prop:concentration-subexponential}
, $\Esp_{\bthetas}[Z|\bzs] \leq Q^2 \Diam(\bTheta) \sqrt{n^2 \nu^2}$ and 
$Z$ is sub-exponential with parameters $(\n^2 \Diam(\bTheta)^2 (2\sqrt{2})^2\nu^2, 2\sqrt{2}\Diam(\bTheta) / b)$. In particular, for all $\vareps_{\n} < \nu b$
\begin{multline*}
  \Prob_{\bthetas}\left( \left. Z \geq \nu \q^2 \Diam(\bTheta) n \left\{ 1 + \frac{\sqrt{8\n^2} \vareps_{\n}}{\q^2} \right\} \right| \bzs \right) \\
  \leq  \Prob_{\bthetas}\left( \left. Z \geq \Esp_{\bthetas}[Z|\bzs] + \nu \Diam(\bTheta) \n^2 2\sqrt{2}\vareps_{\n} \right| \bzs \right) \\  \leq \exp \left( - \frac{\n^2\vareps^2_{\n}}{2} \right)
\end{multline*}
We can then remove the conditioning and take a union bound.
\end{proof}

\subsection{Proof of Proposition~\ref{prop:large-deviations-profile-likelihood} (contribution of far away assignments)}

\begin{proof}
Conditionally on $\bzs$, we know from
proposition~\ref{prop:maximum-conditional-likelihood} that
$\Lambtilde$ is maximal in $\bzs$ and its equivalence
class. Choose $0 < t_n$ decreasing to $0$ but satisfying $\frac{n \rho t_n}{\sqrt{\log(n)}} \to +\infty$. According to \ref{prop:maximum-conditional-likelihood} (iii), for all $\bz \notin S(\bzs, t_n)$
\begin{equation}
\label{eqn:maxlambdatilde}
\Lambtilde(\bz) \leq - c\rho\n\frac{ 3\delta(\bals)}{4} \|\bz - \bzs\|_{0, \sim} \leq - c\rho\frac{ 3\delta(\bals)}{4} \n^2 t_n
\end{equation}
since $\|\bz - \bzs\|_{0, \sim} \geq \n \tn$. 

Set $\vareps_{n} = \inf( 5c\rho\delta(\bals) \tn / (\sqrt{2}\nu\Diam(\bTheta)) , \nu b)$ and $n$ large enough that $\epsilon_{n}\geq\frac{Q^2}{n\sqrt{8}}$. By proposition~\ref{prop:conditional-likelihood-convergence}, and with our choice of $\vareps_{n}$, 
with probability higher than $1 - \Delta_{\n}^1(\vareps_{\n})$,
\begin{align*}
  &\sum_{\bz \notin S(\bzs, \tn)} \prob(\byo, \bz; \btheta)\\ 
  & = \prob(\byo| \bzs, \bthetas) \sum_{\bz \notin S(\bzs, \tn)} \prob(\bz; \btheta) e^{\Fn(\btheta, \bz) - \Lambtilde(\bz) + \Lambtilde(\bz)} \\
  & \leq \prob(\byo| \bzs, \bthetas) \sum_{\bz} \prob(\bz; \btheta) e^{\Fn(\btheta, \bz) - \Lambtilde(\bz) - 3\n^2 \tn c\rho\delta(\bals) / 4} \\
  & \leq \prob(\byo| \bzs, \bthetas) \sum_{\bz} \prob(\bz; \btheta) e^{-\n^2 \tn c\rho\delta(\bals) / 8} \\
  & = \frac{\prob(\byo, \bzs; \bthetas)}{\prob(\bzs; \bthetas)} e^{- \n^2 \tn c\rho\delta(\bals) / 8} \\
  & \leq \prob(\byo, \bzs; \bthetas) \exp\left( -\n^2 \tn \frac{c\rho\delta(\bals)}{8} + \n\log \frac{1}{c} \right) \\
  & = \prob(\byo, \bzs; \bthetas) \smallO(1)
\end{align*}
where the second line comes from inequality (\ref{eqn:maxlambdatilde}), the third from the global control studied in Proposition~\ref{prop:conditional-likelihood-convergence} and the definition of $\vareps_{\n}$, the fourth from the definition of $\prob(\byo, \bzs; \bthetas)$, the fifth from the bounds on $\bpis$ and the last from $\frac{\n\rho\tn}{\sqrt{\log(n)}} \to +\infty$. 

In addition, with our choice of $\tn$, we have $\vareps_{n} \gg \sqrt{\log(\n)} / \n$ so that the series $\sum_{\n} \Delta_{\n}^1(\vareps_{\n})$ converges and:
\begin{align*}
  \sum_{\bz \notin S(\bzs, \tnd)} \prob(\byo, \bz; \btheta) & = \prob(\byo; \bzs, \bthetas) \smallO_P(1)
\end{align*}

\end{proof}  

\subsection{Proof of Proposition~\ref{prop:profile-likelihood-convergence-local} (local convergence $\Fn $)} 

\begin{proof}

We work conditionally on $\bzs \in \mcZ_1$. Choose $\vareps \leq \neighborsize \vmin^2$ small. Assignments $\bz$ at $\|.\|_{0,\sim}$-distance less than $c/4$ of $\bzs$ are $c/4$-regular. According to Proposition B.1 of \cite{BraultKeribinMariadassou2017}
%~\ref{proposition:maxzw}
, $\hykl$ and $\barykl$ are at distance at most $\vareps$ with probability higher than $1 - 2\exp \left( - \frac{\n^2 c^2 \vareps^2}{32(\nu^2 + b^{-1} \vareps)}\right)$. Defining 
\begin{equation*}
 \tilde{\Lambtilde}(\bz) = \sum_{\kk} \sum_{\el} N_{q\el}^o(\bz) \nu(\baryklz, \alkl)  - \sum_{\kk} \sum_{\el} N_{q\el}^o(\bz^\star) \nu(\normp(\alskl), \alskl),
\end{equation*}
where $\Lambtilde(\bz) = \mathbb{E}\left[ \tilde{\Lambtilde}(\bz) | \bz^\star \right]$. Manipulation of $\Lamb$, $\Lambtilde$ and $\tilde{\Lambtilde}$ yield
\begin{align*}
%  \frac{\Fn(\btheta, \bz) - \Lambtilde(\bz)}{\n^2} \leq &
%\frac{\Lamb(\bz) - \Lambtilde(\bz)}{\n^2} \\
 \frac{\Lamb(\bz) - \Lambtilde(\bz)}{\n^2}  \leq & \frac{\Lamb(\bz) - \tilde{\Lambtilde}(\bz)}{\n^2} + \frac{\tilde{\Lambtilde}(\bz) - \Lambtilde(\bz)}{\n^2} \\
   = & \frac{1}{n^2}\sum_{\kk} \sum_{\el} \left( N_{q\el}^o(\bz) \left[ f(\hy_{\kk\el}) - f(\bary_{\kk\el}) \right] - N_{q\el}^o(\bz^\star) \alskl (\hykl^\star - \barykl^\star) \right)\\
   & + \frac{1}{n^2}\sum_{\kk} \sum_{\el}  f(\barykl^\star)\underbrace{\left[ N_{q\el}^o(\bz) - N_{q\el}^o(\bz^\star) - \rho(N_{q\el}(\bz) - N_{q\el}(\bz^\star)) \right]}_{=A_{q\ell}}  \\
   & + \frac{1}{n^2}\sum_{\kk} \sum_{\el} [ N_{q\el}^o(\bz) - \rho N_{q\el}(\bz)]( f(\bary_{\kk\el}^\star) - f(\bary_{\kk\el})) 
\end{align*}
where $f(x) =  x\normpm(x) - \norm\circ\normpm(x)$, $\hykl^\star = 
\hykl(\bzs)$ and $\barykl^\star = \normp(\alskl)$. 
\paragraph{Concerning the first term.}
The function $f$ is twice differentiable on $\mathring{\mathcal{A}}$ with $f'(x) = 
\normpm(x)$ and $f''(x) = 1 / \norm'' \circ \normpm(x)$. $f'$ (resp. $f''$)
are bounded over $I = \normp(C_{\pi})$ by $C_{\al}$ 
(resp. $1/\vmin^2$) so that:
\begin{equation*}
  f(\hykl) - f(\barykl) = f'(\barykl) \left( \hykl - \barykl \right) + 
\Omega\left( (\hykl - \barykl)^2 \right)
\end{equation*}
By Proposition B.1 (adapted for SBM) of \cite{BraultKeribinMariadassou2017} %~\ref{proposition:maxzw}
, $(\hykl - \barykl)^2 = 
\bigO_P(1/\n^2)$ where the $\bigO_P$ is uniform in $\bz$ and does not 
depend on $\bzs$. Similarly, 
\begin{equation*}
  f'(\barykl) = f'(\barykl^\star) + \Omega(\barykl - \barykl^\star) = \alskl 
+ \Omega(\barykl - \barykl^\star)
\end{equation*}
$\barykl$ is a convex combination of the $\Salkl = \normp(\alskl)$ therefore, 
\begin{align*}
  | \barykl - \barykl^\star | &= \left| \frac{\left[\RQbz\transpose \Sal \RQbz 
\right]_{\kk\el}}{\pizk\pizl} - \barykl^\star \right| \\
  & \leq \left( 1 - \frac{\RQbz_{\kk\kk}\RQbz_{\el\el}}{\pizk\pizl}\right) 
(\Salmax - \Salmin)
\end{align*}
Note that: 
\begin{align*}
  \sum_{\kk,\el} N_{q\ell}^o(\bz) \left( 1 - 
\frac{\RQbz_{\kk\kk}\RQbz_{\el\el}}{\pizk\pizl}\right) &= n^2\rho(1 + \smallO_P(1))\sum_{\kk, \el} [1-\RQbz_{\kk\kk}\RQbz_{\el\el}] \\ 
 &= n^2\rho(1 + \smallO_P(1))[1 - \Trace(\RQbz)^2] \\ 
 & \leq n\rho(1 + \smallO_P(1)) 2\|\bz - \bzs\|_{0, \sim}
\end{align*}
and $\hykl - \barykl = \smallO_P(1)$. Therefore 
\begin{equation*}
\frac{1}{n^2}\sum_{\kk,\el} N_{q\ell}^o(\bz) \Omega(\barykl - \barykl^\star) \times (\hykl 
- \barykl) = \smallO_P\left( \frac{\|\bz - \bzs\|_{0, \sim}}{\n}\right)
\end{equation*}
The remaining term writes
\begin{align*}
\frac{1}{n^2}\sum_{\kk,\el} \alskl \left[ N_{q\ell}^o(\bz) (\hykl - \barykl) - N_{q\ell}^o(\bz^\star) 
(\hykl^\star - \barykl^\star)  \right]
\end{align*}
and is also $\smallO_P(\left( \|\bz - \bzs\|_{0, \sim} / \n \right)$ uniformly in $\bz$ and $\bzs \in 
\Omega_1$ by Proposition~\ref{proposition:maxdiffz}. 

\paragraph{Concerning the second term.}

For all $q,\ell$, defining 

$$
\left\{
    \begin{array}{ll}
        N_{q\ell}^{+}(\bz, \bzs) = n^2\sum_{q^\prime}\RQbz_{\kk\kk^\prime}(\bz)  \sum_{\ell^\prime}\RQbz_{\ell\ell^\prime}(\bz)  - n^2\RQbz_{\kk\kk} \RQbz_{\el\el}\\
        N_{q\ell}^{-}(\bz, \bzs) = n^2\sum_{q}\RQbz_{\kk\kk^\prime}(\bz)  \sum_{\ell}\RQbz_{\ell\ell^\prime}(\bz)  - n^2\RQbz_{\kk\kk} \RQbz_{\el\el}
    \end{array}
\right.
$$
and noticing that $N_{q\ell}^{+}(\bz, \bzs) = \#\{ (i,j) : z_{iq}=1, z_{j\ell}=1, (z_{q\el},z_{j\el}) \neq (z_{q\el}^\star,z_{j\el}^\star)\}$ and $N_{q\ell}^{-}(\bz, \bzs) = \#\{ (i,j) : z_{iq}^\star=1, z_{j\ell}^\star=1, (z_{q\el},z_{j\el}) \neq (z_{q\el}^\star,z_{j\el}^\star)\}$. Using the following notations

\begin{equation*}
 \hat\rho_{q\ell}^{+} = \frac{1}{N_{q\ell}^{+}(\bz, \bzs)} \sum_{(i,j) \in N_{q\ell}^{+}(\bz, \bzs)}R_{ij}, \quad \hat\rho_{q\ell}^{-} = \frac{1}{N_{q\ell}^{-}(\bz, \bzs)} \sum_{(i,j) \in N_{q\ell}^{-}(\bz, \bzs)}R_{ij}
\end{equation*}

we are able to write
\begin{align*}
 A_{q\ell} = & \sum\limits_{\substack{i<j \\ z_{iq}=1, z_{j\ell}=1}}(R_{ij}-\rho) - \sum\limits_{\substack{i<j \\ z_{iq}^\star=1, z_{j\ell}^\star=1}} (R_{ij}-\rho) \\
 = & N_{q\ell}^{+}(\bz, \bzs)(\hat\rho_{q\ell}^{+} - \rho) - N_{q\ell}^{-}(\bz, \bzs)(\hat\rho_{q\ell}^{-} - \rho).
\end{align*}
Where the second equality is the sum of independent random variables.  \\
Note that :

\begin{align*}
  \sum_{q\el} N_{q\ell}^{+}(\bz, \bzs) &= \sum_{q\el} N_{q\ell}^{-}(\bz, \bzs) \\
  &= n^2\sum_{\kk, \el} [1-\RQbz_{\kk\kk}\RQbz_{\el\el}] \\ 
 &= n^2[1 - \Trace(\RQbz)^2] \\ 
 & \leq n 2\|\bz - \bzs\|_{0, \sim}
\end{align*}

also that $\hat\rho_{q\ell}^{+}-\rho = \smallO_P\left(1\right)$ and $\hat\rho_{q\ell}^{-}-\rho = \smallO_P\left(1\right)$. Therefore
\begin{equation*}
 \frac{1}{n^2}\sum_{\kk} \sum_{\el}  f(\barykl^\star)A_{q\el} = \smallO_P\left( \frac{\|\bz - \bzs\|_{0, \sim}}{\n}\right).
\end{equation*}

\paragraph{Concerning the third term.}
Using arguments developed previously leads to the same conclusion than before :
\begin{equation*}
 \frac{1}{n^2}\sum_{\kk} \sum_{\el} [ N_{q\el}^o(\bz) - \rho N_{q\el}(\bz)]( f(\bary_{\kk\el}^\star) - f(\bary_{\kk\el})) = \smallO_P\left( \frac{\|\bz - \bzs\|_{0, \sim}}{\n}\right).
\end{equation*}

As a conclusion, writing 
\begin{equation*}
\sup_{\bz \nsim \bzs} \frac{\Lamb(\bz) - 
\Lambtilde(\bzs)}{\n \|\bz - \bzs\|_{0, 
\sim}} = \sup_{\bz \nsim \bzs} \left( \frac{\Lamb(\bz) - \Lambtilde(\bz) }{n \|\bz - \bzs\|_{0, 
\sim}} + \frac{\Lambtilde(\bz) - \Lambtilde(\bzs)}{n \|\bz - \bzs\|_{0, 
\sim}} \right)
\end{equation*}
and noticing that $\frac{\Lambtilde(\bz) - \Lambtilde(\bzs)}{n \|\bz - \bzs\|_{0, 
\sim}} \leq 0$ since $\Lambtilde$ is maximized in $\bzs$ (see \ref{prop:maximum-conditional-likelihood}). We have
\begin{equation*}
\sup_{\bz \nsim \bzs} \frac{\Lamb(\bz) - 
\Lambtilde(\bzs)}{\n \|\bz - \bzs\|_{0, 
\sim}} = \smallO_P(1).
\end{equation*}

\end{proof}

\subsection{Proof of Proposition~\ref{prop:small-deviations-profile-likelihood} (contribution of local assignments)} 

\begin{proof}
By Proposition~\ref{cor:prob-regular-configurations-star}, it is enough to prove 
that the sum is small compared to $\prob(\bzs, \byo; \bthetas)$ on $\Om_1$. 
We work conditionally on $\bzs \in \mcZ_1$. Choose $\bz$ in $S(\bzs, C)$ with $C$ defined in 
proposition~\ref{prop:large-deviations-profile-likelihood}. 
\begin{align*}
  \log \left( \frac{\prob(\bz, \byo; \btheta)}{\prob(\bzs, \byo; 
\bthetas)} \right) & = \log \left( \frac{\prob(\bz; \btheta)}{\prob(\bzs; \bthetas)}\right) + \Fn(\btheta, \bz) \\
\end{align*}
For $C$ small enough, we can assume without loss of generality that $\bz$ 
is the representative closest to $\bzs$ and note $r = \|\bz - 
\bzs\|_0$. Then:
\begin{align*}
  \Fn(\btheta, \bz) & \leq \Lamb(\bz) - \Lambtilde(\bz) + 
\Lambtilde(\bz) \\
  & \leq \Lamb(\bz) - \Lambtilde(\bz) - 
c\rho\frac{3\delta(\bals)}{4}\n r \\
  & \leq c\rho\frac{3\delta(\bals)}{4}\n r(1 + 
\smallO_P(1)) \\
\end{align*}
where the first line comes from the definition of $\Lamb$, the second line from 
Proposition~\ref{prop:maximum-conditional-likelihood} and the third from 
Proposition~\ref{prop:profile-likelihood-convergence-local}. Thanks to 
proposition~\ref{cor:marginalprobabilties}, we also know that:
\begin{equation*}
 \log \left( \frac{\prob(\bz; \btheta)}{\prob(\bzs; \bthetas)}\right) 
\leq \bigO_P(1) \exp\left\{M_{c/4}r \right\}
\end{equation*}
There are at most ${\n \choose r} Q^{r}$ assignments 
$\bz$ at distance $r$ of $\bzs$ and each of them has 
at most $Q^{Q}$ equivalent configurations. Therefore, 
\begin{align*}
 & \frac{\sum_{\substack{\bz \in S(\bzs, \tilde{c}) \\ \bz 
\nsim \bzs}} \prob(\bz, \byo; \btheta)}{\prob(\bzs, \byo; 
\bthetas)} \\
  & \leq \bigO_P(1) \sum_{\substack{r \geq 1}} {\n \choose r}\Q^{\Q + r} \exp\left( r M_{c/4} - c \rho\frac{3\delta(\bals)}{4}\n r(1 + \smallO_P(1)) \right) 
\\
  & = \bigO_P(1) \left(1+ e^{(\Q+ 1)\log \Q + M_{c/4} - c\rho\n \frac{3  
\delta(\bals)(1 + \smallO_P(1))}{4} }\right)^{\n} -1 \\
  & \leq \bigO_P(1) a_{\n} \exp(a_{\n})
\end{align*}
where $a_{\n} = \n e^{(\Q + 1)\log \Q + M_{c/4} - c\rho\n \frac{3 
\delta(\bals)(1 + \smallO_P(1))}{4} } = \smallO_P(1)$.

\end{proof}
 
\subsection{Proof of Proposition~\ref{prop:equivalent-configurations-profile-likelihood} (contribution of equivalent assignments)} 

\begin{proof}

Choose $s$ permutations of $\{1, \dots, \Q\}$ and assume that $\bz = \bz^{\vrai,s}$. Then $\prob(\byo, \bz; \btheta) = \prob(\byo, \bz^{\vrai, s}; \btheta) = \prob(\byo, \bzs; \btheta^{s})$.  If furthermore $s \in \Symmetric(\btheta)$, $\btheta^{s} = \btheta$ and immediately $\prob(\byo, \bz; \btheta) = \prob(\byo, \bzs; \btheta)$. We can therefore partition the sum as 

\begin{align*}
  \sum_{\bz \sim \bz^\star} \prob(\byo, \bz; \btheta) & = \sum_{s} \prob(\byo, \bz^{\vrai, s}; \btheta) \\ 
  & = \sum_{s} \prob(\byo, \bzs; \btheta^{s}) \\ 
  & = \sum_{\btheta' \sim \btheta} \# \Symmetric(\btheta') \prob(\byo, \bzs; \btheta') \\ 
  & = \# \Symmetric(\btheta) \sum_{\btheta' \sim \btheta} \prob(\byo, \bzs; \btheta') \\ 
\end{align*}

$\prob(\byo, \bzs; \btheta)$ unimodal in $\btheta$, with a mode in $\bthetaMC$. By consistency of $\bthetaMC$, either $\prob(\byo, \bzs; \btheta) = \smallO_P(\prob(\byo, \bzs; \bthetas))$ or $\prob(\byo, \bzs; \btheta) = \bigO_P(\prob(\byo, \bzs; \bthetas))$ and $\btheta \to \bthetas$. In the latter case, any $\btheta' \sim \btheta$ other than $\btheta$ is bounded away from $\bthetas$ and thus $\prob(\byo, \bzs; \btheta') = \smallO_P(\prob(\byo, \bzs; \bthetas))$. In summary, 
\[
 \sum_{\btheta' \sim \btheta} \frac{\prob(\byo, \bzs; \btheta')}{\prob(\byo, \bzs; \bthetas)} = \max_{\btheta' \sim \btheta} \frac{\prob(\byo, \bzs; \btheta')}{\prob(\byo, \bzs; \bthetas)} (1 + \smallO_P(1))
\]

\end{proof}

\subsection{Proof of Corollary~\ref{cor:behaviorEMV}: Behavior of $\bthetaEMV$}
\label{annexe:cor:behaviorEMV}

We may prove the corollary by contradiction. Note first that 
unless $\bTheta$ is constrained and with high probability, $\bthetaEMV$ and 
$\hbtheta(\bzs)$ exhibit no symmetries. Indeed, equalities like $\hykl 
= \hy_{\kp, \lp}$ have vanishingly small probabilities of being 
simultaneously true when $y_{ij}$ is discrete, and even null when $y_{ij}$ is 
continuous. Assume then $\min_{s} 
(\hbpi_{MLE}^{s} - \hat{\bpi}\left(\bzs\right)) \neq 
\smallO_P\left(1/\sqrt{\n}\right)$ or 
$\min_{s} (\hbal_{MLE}^{s} - \hat{\bal}\left(\bzs\right)) \neq 
\smallO_P\left(1/\n\right)$ where $s$ is a 
permutation of $\{1, \dots, \Q\}$. Then, 
by Proposition~\ref{prop:LocalAsymp} and the consistency of 
$\hat{\btheta}\left(\bzs\right)$
\begin{equation}\label{eq:proof:ConsistenceEMV}
\min_{s} \Lcs\left(\hat{\btheta}\left(\bzs\right)\right)- 
\Lcs\left(\bthetaEMV^{s}\right)=\Omega_{P}(1).
\end{equation}
But, since $\hat{\btheta}\left(\bzs\right)$ and $\bthetaEMV$ maximise 
respectively $\frac{\prob(\byo, \bzs; \btheta')}{\prob(\byo, \bzs; 
\bthetas)}$ and $\frac{\prob(\byo; \btheta)}{\prob\left(\byo; 
\btheta^{\vrai}\right)}$ and have no symmetries, it follows by 
Theorem~\ref{thm:observed-akin-to-complete-general} that 
\[\left|\frac{\prob\left(\byo, \bzs; 
\hat{\btheta}\left(\bzs\right)\right)}{\prob(\byo, \bzs; \bthetas)}- 
\max_{s} \frac{\prob\left(\byo, \bzs; 
\bthetaEMV^{s}\right)}{\prob(\byo, \bzs; \bthetas)}\right|=\smallO_P(1)\]
which contradicts Eq~(\ref{eq:proof:ConsistenceEMV}) and concludes the proof.

\subsection{Proof of Corollary~\ref{cor:Variational}: Behavior of $\Jvar\left(\setQ,\btheta\right)$}\label{annexe:cor:Variational}

Remark first that for every $\btheta$ and for every $\bz$, 
\[\prob\left(\byo,\bz;\btheta\right)\leq\exp\left[\Jvar\left(\delta_{\bz},\btheta\right)\right]
\leq 
\underset{\setQ\in\mcQ}{\max}\;\exp\left[\Jvar\left(\setQ,\btheta\right)\right]
\leq\prob\left(\byo;\btheta\right)\]
where $\delta_{\bz}$ denotes the dirac mass on $\bz$. By dividing by 
$\prob\left(\byo;\bthetas\right)$, we obtain
\[\frac{\prob\left(\byo,\bz;\btheta\right)}{\prob\left(\byo;\bthetas\right)}
\leq 
\frac{\underset{\setQ\in\mcQ}{\max}\;\exp\left[\Jvar\left(\setQ,
\btheta\right)\right]}{\prob\left(\byo;\bthetas\right)} 
\leq\frac{\prob\left(\byo;\btheta\right)}{\prob\left(\byo;\bthetas\right)}.\]
As this inequality is true for every couple $\bz$, we have in particular:
\[
\underset{\bz \sim 
\bzs}{\max}\frac{\prob\left(\byo,\bz;\btheta\right)}{ 
\prob\left(\byo;\bthetas\right) } = \underset{\btheta' \sim \btheta}{\max} 
\frac{\prob\left(\byo,\bzs;\btheta'\right)}{ 
\prob\left(\byo;\bthetas\right) }\leq 
\frac{\underset{\setQ\in\mcQ}{\max}\;\exp\left[\Jvar\left(\setQ, 
\btheta\right)\right]}{\prob\left(\byo;\bthetas\right)} .\]
Noticing that 
$\prob\left(\byo;\bthetas\right) = \#\Symmetric(\bthetas)\prob\left(\byo,
\bzs;\bthetas\right)(1+o_p(1))$, 
Theorem~\ref{thm:observed-akin-to-complete-general} therefore leads to the 
following bounds:
\begin{multline*}
\underset{\btheta' \sim \btheta}{\max} 
\frac{\prob\left(\byo,\bzs;\btheta'\right)}{ 
\prob\left(\byo, \bzs;\bthetas\right)}(1+\smallO_P(1)) \leq 
\frac{\underset{\setQ\in\mcQ}{\max}\;\exp\left[\Jvar\left(\setQ, 
\btheta\right)\right]}{\prob\left(\byo, \bzs;\bthetas\right)} \\
\leq 
\#\Symmetric(\btheta) \underset{\btheta' \sim \btheta}{\max} 
\frac{\prob\left(\byo,\bzs;\btheta'\right)}{\prob\left(\byo, \bzs ;\bthetas\right)}(1+\smallO_P(1)) + \smallO_P(1).
\end{multline*}
Again unless $\bTheta$ is constrained, $\bthetaVAR$ exhibits no symmetries with 
high probability and the same proof by contradiction as in 
appendix~\ref{annexe:cor:behaviorEMV} gives the result.\\

% flatex input end: [appendix_proof-main-results.tex]

%% Appendices et preuves
\section{Sub-exponential random variables}
% flatex input: [appendix_sub-exponential-rv.tex]
\label{appendix:sub-exponential}

%% Subexponential variables
We now prove two propositions regarding subexponential variables. 
Recall first that a random variable $X$ is  sub-exponential with parameters $(\tau^2, b)$ if for all $\lambda$ such that $|\lambda|\leq 1/b$,
\[
\Esp[e^{\lambda( X-\Esp(X))}] \leq \exp\left( \frac{\lambda^2 \tau^2}{2} \right).
\]
In particular, all distributions coming from a natural exponential family are sub-exponential. Sub-exponential variables satisfy a large deviation Bernstein-type inequality:
\begin{equation}
  \label{eq:concentration-subexponential}
  \Prob( X - \Esp[X] \geq t) \leq 
  \begin{cases} 
    \exp\left( - \frac{t^2}{2 \tau^2}\right)  & \text{if} \quad 0 \leq t \leq \frac{\tau^2}{b} \\ 
    \exp\left( - \frac{t}{2b}\right)  & \text{if} \quad t \geq \frac{\tau^2}{b}
  \end{cases}
\end{equation}
So that
\[
\Prob( X - \Esp[X] \geq t) \leq \exp\left( - \frac{t^2}{2(\tau^2+bt)}\right)
\]
The subexponential property is preserved by summation and multiplication. 
\begin{itemize}
\item If $X$ is sub-exponential with parameters $(\tau^2, b)$ and $\alpha \in \mathbb{R}$, then so is $\alpha X$ with parameters $(\alpha^2\tau^2, \alpha b)$
\item If the $X_i$, $i = 1,\dots,n$ are sub-exponential with parameters $(\tau_i^2, b_i)$ and independent, then so is $X = X_1 + \dots + X_n$ with parameters $(\sum_i \tau_i^2,\max_i b_i)$
\end{itemize}

\begin{thm}[Equivalent characterizations of sub-exponential variables]\label{thm:subexp_eq}
 For a zero-mean random variable $X$, the following statements are equivalent:
 \begin{enumerate}
  \item There are non-negative numbers $(\nu,b^{-1})$ such that $$\Esp[e^{\lambda X}] \leq \exp\left( \frac{\lambda^2 \nu^2}{2} \right) \quad \text{for all } |\lambda|<b.$$
  \item There is a positive number $c_0>0$ such that $\Esp[e^{\lambda X}]<\infty$ for all $|\lambda|<c_0$.
  \item There are constants $c_1,c_2>0$ such that $$ \mathbb{P}(|X|\geq t) \leq c_1 e^{-c_2 t} \quad \text{for all } t>0.$$
  \item The quantity $\gamma:=\sup_{k\geq2}\left[ \frac{\mathbb{E}[X^k]}{k!} \right]^{1/k}$ is finite.
 \end{enumerate}
\end{thm}
\begin{proof}
 A proof of this theorem can be found in \cite{wainwright}.
\end{proof}

\begin{proposition}[Maximum in $\bz$] \label{proposition:maxdiffz}
  Let $(\bar{\bz}$ be any configuration and $\bz$ the 
$\sim$-equivalent configuration that achieves $\|\bz - \bzs\|_0 = \|\bar{\bz} - 
\bzs\|_{0,\sim}$ let $\hykl = \hat{y}_{\kk, \el}(\bz)$ (resp. 
$\barykl(\bz)$) and $\hykl^\star = \hat{y}_{\kk, \el}(\bzs)$ (resp. $\barykl^\star = \barykl(\bzs)$ = $\normp(\alskl)$) be as 
defined in Equations~\eqref{eq:mle-complete-likelihood} 
and~\eqref{eq:expectation-ychap}. Under the assumptions of the 
section~\ref{sec:assumptions}, for all $\vareps \leq \neighborsize 
\vmax^2$, 
  \begin{multline*}
    \label{eq:ineg-de-bernst-1}
    \Prob \left( \max_{\bar{\bz} \nsim \bzs} \max_{k,l} 
\frac{ N^{o}_{q\el}(\bz) (\hat{y}_{\kk, \el} - \barykl) - N^{o}_{q\el}(\bzs) 
(\hykl^\star - \barykl^\star ) }{\n\|\bz - \bzs\|_{0}} > \vareps \right) = \smallO(1)
  \end{multline*}
\end{proposition}

\begin{proof}
Note $r = \|\bz - \bzs\|_{0}$. The numerator within the $\max$ in the fraction 
can be expanded to 
\begin{equation*}
Z_{\kk\el}(\bz) = \sum_{i, j} (\zik\zjl - \zik^\star \zjl^\star) (y_{ij}r_{ij} 
- \als_{\zik^\star \zjl^\star}\rho)
\end{equation*}
and is thus a sum of at most $N = \n r$ 
non-null centered subexponential random variables with parameters $(a^2, 
1/w)$. It is therefore a centered subexponential with 
parameters $(N a^2, 1/w)$. By Bernstein inequality, for all 
$\vareps \leq \neighborsize a^2$ we have
\begin{equation*}
\Prob (Z \geq \vareps \n r ) \leq \exp\left(- 	
\frac{\n r\vareps^2}{2a^2} \right).
\end{equation*}
There are at most $\n^{r} \Q^{r} \Q^\Q$ $\bz$ at $\|.\|_{0,\sim}$ distance 
$r$ of $\bzs$. An union bound shows that:
\begin{multline*}
\Prob \left( \max_{\bar{\bz} \nsim \bzs} \max_{q,\ell} 
\frac{Z_{\kk\el}(\bz)}{\n\|\bz - \bzs\|_{0}} 
 \geq \vareps \right) \\ 
 \leq \sum_{r \geq 1} \sum_{\substack{r=\|\bar{\bz} - \bzs\|_{0, 
\sim}}{}} \Q^2 \Prob (Z_{\kk\el}(\bz) \geq 
\vareps \n r ) \\
 \leq \sum_{r \geq 1} \Q^\Q \exp\left(-\n r\vareps^2 / 2 a^2 + r \log(\n\Q) + 2 \log(\Q)\right) = o(1)
\end{multline*}
where the last equality is true as soon as $\n \vareps_{\n} \gg \log\n$.

\end{proof}

\section{Likelihood ratio of assignments}
% flatex input: [appendix_likelihood-ratio.tex]
\label{appendix_likelihood-ratio}

\begin{proposition}\label{cor:marginalprobabilties}\textbf{}
Let $\bzs$ be $c/2$-regular and $\bz$ at $\|.\|_0$-distance $c/4$ of $\bzs$. Then, for all $\btheta \in \bTheta$
\begin{equation*}
\log \frac{\prob(\bz; \btheta)}{\prob(\bzs; \bthetas)} \leq \bigO_P(1) \exp \left\{ M_{c/4}  \|\bz - \bzs \|_0 \right\} \\ 
\end{equation*}
\end{proposition}

\proofbegin
Note then that:
\begin{align*}
\frac{\prob(\bz; \btheta)}{\prob(\bzs; \bthetas)} & = & \frac{\prob(\bz; \bpi)}{\prob(\bzs; \bpis)} = 
\frac{\prob(\bz; \bpi)}{\prob(\bzs; \hat{\bpi}(\bzs))} \frac{\prob(\bzs; \hat{\bpi}(\bzs))}{\prob(\bzs; \bpis)} \\
& \leq & \frac{\prob(\bz; \hat{\bpi}(\bz))}{\prob(\bzs; \hat{\bpi}(\bzs))} \frac{\prob(\bzs; \hat{\bpi}(\bzs))}{\prob(\bzs; \bpis)} \\
& \leq & \exp \left\{ M_{c/4} \|\bz - \bzs \|_0 \right\} \times \frac{\prob(\bzs; \hat{\bpi}(\bzs))}{\prob(\bzs; \bpis)} \\
& \leq & \bigO_P(1) \exp \left\{ M_{c/4}  \|\bz - \bzs \|_0  \right\}  \\
\end{align*}
where the first inequality comes from the definition of $\hat{\bpi}(\bz)$ and the second from Lemma B.6 of \cite{BraultKeribinMariadassou2017} %~\ref{lemme:marginalprobabilties} 
and the fact that $\bzs$ and $\bz$ are $c/4$-regular. 
Finally, local asymptotic normality of the MLE for multinomial proportions ensures that $\frac{\prob(\bzs; \hat{\bpi}(\bzs))}{\prob(\bzs; \bpis)} = \bigO_P(1)$. 
\proofend

% flatex input end: [appendix_likelihood-ratio.tex]

%% Appendices et preuves
% \section{Non-modified technical results}
% \input{appendix_resultsRecopied.tex}

%*flatex input: [SBM-MCAR.bbl]

% flatex input end: [SBM-MCAR.bbl]
%FLATEX-REM:\bibliographystyle{abbrvnat} % plain
%FLATEX-REM:\bibliography{bibliographie_theoreticalPaper}

\end{document}